\documentclass[a4paper,10pt,reqno]{amsart}
\usepackage{amssymb,latexsym,bbm,amsmath,amsfonts}
\usepackage{amsthm}
\usepackage[mathscr]{eucal}
\usepackage{rotating}
\usepackage{multirow}
\usepackage{slashbox}

\numberwithin{equation}{section}
\newtheorem{theorem}{Theorem}
\newtheorem{lemma}{Lemma}

\theoremstyle{definition}
\newtheorem{definition}[theorem]{Definition}

\theoremstyle{remark}

\begin{document}
	
	\title[Perturbed Linear ODEs and SDEs]
	{Weighted $L_\infty$ Asymptotic Characterisation of Perturbed Autonomous Linear Ordinary and Stochastic Differential Equations: Part I - ODEs}
	
	\author{John A. D. Appleby}
	\address{School of Mathematical
		Sciences, Dublin City University, Glasnevin, Dublin 9, Ireland}
	\email{john.appleby@dcu.ie} 
	
	\author{Emmet Lawless}
	\address{School of Mathematical
		Sciences, Dublin City University, Glasnevin, Dublin 9, Ireland}
	\email{emmet.lawless6@mail.dcu.ie}
	
	\thanks{EL is supported by Science Foundation Ireland (16/IA/4443). JA is supported by the RSE Saltire Facilitation Network on Stochastic Differential Equations: Theory, Numerics and Applications (RSE1832).} 
	\subjclass{}
	\keywords{}
	\date{21 October 2024}
	
	\begin{abstract}
		This is the first of a two--part paper which determines necessary and sufficient conditions on the asymptotic behaviour of forcing functions so that the solutions of additively pertubed linear differential equations obey certain growth or decay estimates. Part I considers deterministic equations, and part II It\^o--type stochastic differential equations. Results from part I are used to deal with deterministically and stochastically forced equations in the second part. Results apply to both scalar and multi--dimensional equations, and connect the asymptotic behaviour of time averages of the forcing terms on finite intervals with the growth or decay rate of the solution. Mainly, results deal with large perturbations, but some indications of how results extend to tackle subdominant perturbations are also sketched.   	
	\end{abstract}
	
	\maketitle
	
	\section{Overview and Scope}
	The study of the asymptotic behaviour of perturbed differential systems is an important and classical topic in the theory of differential equations, and the most basic results concern perturbations from homogeneous autonomous linear differential equations e.g., 
	\[
	u'(t)=Au(t), \quad t\geq 0
	\] 
	to form the perturbed ordinary differential equation
	\begin{equation} \label{eq.x1intro}
		x'(t)=Ax(t)+p(t,x(t)), \quad t\geq 0
	\end{equation}
	Here $A$ is a $d\times d$ matrix, and $p:[0,\infty)\times\mathbb{R}^d$ is a suitably regular function. 
	
	The thrust of a great deal of the research falls into two broad categories, and of course is not confined to autonomous equations. First, to give smallness conditions on $p$, of an appropriate character, such that the asymptotic behaviour of $x$ is similar to that of $u$. Results such as the linearisation theorem fall into this category, along with Hartman--Wintner theorems. 
	
	An important second line of investigation considers stable equations $u$ (in which solutions of $u$ tend to zero), and the perturbation lies in a certain ``interesting'' function space (e.g., the space of periodic, or bounded, or integrable, or convergent, functions). In this case, it is typically the case that the solution $x$ will lie in the same space as the perturbation. 
	
	In many ways, the investigations in this paper follows the theme of the second line, and touch a little on the first, with some indications of how work might proceed in the case of unstable equations, or equations with subdominant perturbations. This part of the paper deals with deterministic equations, but it turns out that many ideas, proofs, and the forms of conditions placed on the problem data, carry over to It\^o--type stochastic differential equations as well, and the second part of the paper will be concerned with such stochastic equations. 
	
	We are mostly interested in giving a \textit{characterisation} of rates of decay, or growth, and asymptotic bounds on solutions in terms of perturbations, and we also restrict attention to the case where the perturbation $p$ is \text{state--independent}. To give a bird's--eye view of what is intended, by state--independent perturbations, we mean that the function $p$ in \eqref{eq.x1intro} does not depend on $x$, so we have 
	$p(t,x)=f(t)$, where $f:[0,\infty)\to \mathbb{R}^d$ has suitably regularity properties (for simplicity, we usually take $f$ to be continuous). On the other hand, when talking about an asymptotic characterisation, we have in mind obtaining conditions $(\text{A})$, $(\text{B})$, $(\text{C})$ on $f$ such that 
	\begin{gather*}
		\text{$f$ obeys $(\text{A})$ }  \Longleftrightarrow \limsup_{t\to\infty} \frac{\|x(t)\|}{\gamma(t)}<+\infty, \\
		\text{$f$ obeys $(\text{B})$ }  \Longleftrightarrow \lim_{t\to\infty} \frac{\|x(t)\|}{\gamma(t)}=0, \\
		\text{$f$ obeys $(\text{C})$ }  \Longleftrightarrow \limsup_{t\to\infty} \frac{\|x(t)\|}{\gamma(t)}\in (0,\infty), 
	\end{gather*}
	where $\gamma:[0,\infty)\to (0,\infty)$ is a continuous weight function with a given rate of growth or decay (or is simply the constant unit function, if one is interested in unweighted boundedness, convergence, or non--convergent boundedness). As might be imagined, the art is to connect the asymptotic size of (some easily computed functional of) $f$ with $\gamma$, and this we do. 
	
	In sketching the high--level objective of the work, we hope that the reader, who might be hitherto skeptical that much that is new or is interesting can be stated about such elementary \textit{autonomous linear ordinary differential} equations, may be more open to our approach. One reason for this is that many asymptotic results for forced equations give only \textit{sufficient conditions} for certain types of asymptotic behaviour, or perhaps demonstrate that conditions are sharp by means of counterexamples, whilst we claim that the appropriate conditions on forcing terms can be captured exactly. 
	
	But this is not all. It will turn out that our broad approach, and the conditions on forcing terms that are generated very naturally by it, are also successful obtaining cognate asymptotic characterisations for stochastic differential equations in both a mean, mean--square and pathwise sense, and demonstrating this will be the main task of the second part of this work. 
	
	Finally, by reference to some of our other works (both published and in progress), the reader will see that the functionals of the forcing terms that we consider here in order to characterise certain weighted $L_\infty$ asymptotic behaviour for ordinary differential equations, can be used to characterise general $L^p$, or periodic, or convergent, or time--averaged asymptotic behaviour of a large class of linear functional differential equations subjected to external state--independent forcing. Furthermore, our experience with diverse problems suggests that a characterisation of when solutions of stable autonomous deterministic linear differential systems lie in almost any important function spaces will hinge on the same functionals of $f$ that we consider here. 
	
	In fact, the situation is perhaps even better than this; on the one hand, characterisation results can be established when the perturbing terms are not only deterministic, but stochastic, and one can still proceed when both deterministic and stochastic state-independent terms are simultaneously present. On the other, the equations studied can be those with finite or unbounded memory, or unperturbed equations which are autonomous or non--autonomous. Indeed, a number of these asymptotic characterisation results for what may be fairly viewed as reasonably complicated general high--dimensional systems (e.g., multidimensional stochastic and deterministically forced functional differential equations with unbounded memory) will rely directly on perturbation results for a single (and simplest--imaginable) scalar ordinary differential equation, namely the ODE
	\[
	u'(t)=-u(t), \quad t\geq 0. 
	\]
	Chiefly for this reason (though also on didactic and expository grounds), we start the first part of this paper building up very elementary scalar perturbation results. We hope that by the end of the first part of this paper (and even moreso by the end of the second part), the reader will see that results for very simple scalar ordinary differential equations are not only of some independent interest, but are also instrumental in tackling large classes of more complicated equations.  
	
	Next, we indicate the kind of conditions (A), (B) and (C) on $f$ that are necessary and sufficient for to characterise the size of solutions of the differential equations. Without going into fine details yet, what we find interesting is that the same type of functional for $f$ works for a very broad class of perturbations. In very rough terms indeed, the reader can orient themselves by using the rule of thumb that $\gamma$ should be chosen so that  
	\[
	\left\|\int_{t-1}^t f(s)\,ds\right\| = O(\gamma(t)), \quad t\to\infty
	\] 
	in order to satisfy (A), or  
	\[
	\left\|\int_{t-1}^t f(s)\,ds\right\|= o(\gamma(t)), \quad t\to\infty,
	\]
	in order to satisfy (B), or  
	\[
	c\gamma(t)\leq 	\left\|\int_{t-1}^t f(s)\,ds\right\|
	\leq C\gamma(t)
	\]
	for some $0<c\leq C<+\infty$, in order to satisfy (C). The precise technical conditions are outlined below. 
	
	As already mentioned, in this work we are largely interested in the situation where the perturbation dominates the unperturbed solution, and the unperturbed equation is (exponentially) asymptotically stable. In this setting, a dominant perturbation suggests either making $\gamma$ decay to zero (or grow to infinity) slower than exponentially, or simply be non--decreasing, and this is what we will mostly ask of the weight function. In doing so, we make more precise what we mean by a subexponential weight function, but the conditions that we need to impose on this auxiliary object are not restrictive. 
	
	We also sketch some results which deal with cases when the perturbation is subdominant, or the unperturbed equation is unstable (or both). Our analysis is not as comprehensive, but the preliminary results we obtain suggest that the type of conditions suggested above, which involve the size of moving averages of $f$ on finite intervals, can still be used to characterise certain types of asymptotic behaviour. 
	We leave a more thorough investigation for a later work. 	        
	
	\section{Technical Introduction}
	In this paper, we are concerned in characterising the pointwise asymptotic behaviour of state--independent perturbations of the linear differential equation 
	\[
	u'(t)=-u(t), \quad t\geq 0, \quad u(0)=\xi.
	\]
	The perturbations we consider are \begin{enumerate}
		\item[(i)] Deterministic 
		\begin{equation} \label{eq.y1}
			y'(t)=-y(t)+f(t), \quad t\geq 0,
		\end{equation}
		where $f$ is a continuous function on $[0,\infty)$;
		\item[(ii)] Stochastic 
		\begin{equation} \label{eq.Y01}
			dY_0(t)=-Y_0(t)\,dt+\sigma(t)\,dB(t), \quad t\geq 0, 
		\end{equation}
		where $\sigma$ is a continuous function, and $B$ is a standard one--dimensional Brownian motion;
		\item[(iii)] Stochastic and deterministic
		\begin{equation} \label{eq.Y1}
			dY(t)=\left(-Y(t)+f(t)\right)\,dt+\sigma(t)\,dB(t),
		\end{equation}
		where $f$, $\sigma$ and $B$ are as above.
	\end{enumerate}
	The study of problems (ii) and (iii) will be in the second part of the paper. Since $Y(t)=y(t)+Y_0(t)$ and the size estimates of $y$ and $Y_0$ are worked out in problems (i) and (ii) most of the additional work involved in studying \eqref{eq.Y1} hinges on showing that the deterministic and stochastic components $y$ and $Y_0$ do not ``cancel each other out'', even when $y$ and $Y_0$ have comparable asymptotic upper bounds. 
	
	Once the problem in part (i) is dealt with, we can apply it to more complicated dynamical systems. In Part I of this work, we will consider finite dimensional linear  differential equations; in part II, the mean square of a finite dimensional SDE will be considered. We do not study path--by--path behaviour of multidimensional SDEs in this work. 
	
	To be specific, our ultimate goal is the following. Let $y$ stand for any of the solutions of the above equations, and $\gamma$ be a weight function, with some additional properties, which allows us to cover a wide range of possible asymptotic behaviour (such as growth or slow decay to zero). We then seek identical necessary and sufficient conditions on the forcing functions $f$ and $\sigma$ so that 
	\[
	\limsup_{t\to\infty} \frac{|y(t)|}{\gamma(t)}\in (0,\infty).
	\] 
	In the case of scalar stochastic equations, we wish the asymptotic behaviour to be true almost surely (a.s.), which is to say that for every $\omega\in \Omega^\ast$, where $\Omega^\ast$ is an event with probability one, each trajectory obeys
	\[
	\limsup_{t\to\infty} \frac{|Y(t,\omega)|}{\gamma(t)}\in (0,\infty)
	\] 
	and $\gamma$ is a \textit{deterministic} function which depends on $f$ and $\sigma$. Such a characterisation can be achieved provided the noise term $\sigma$ does not grow more quickly than an exponential function (there are no size restrictions on $f$). In the case when $\sigma$ exhibits faster than exponential growth (and again, we do not restrict $f$), we can give sharp sufficient conditions for the exact fluctuation size of $Y$ in terms of a weight function. These results appear in Part II of the paper.   
	
	Confining so much of our attention to scalar equations whose unperturbed solution decays at the parameter independent rate $e^{-t}$ might seem very limited. However, if we are interested in non--exponentially convergent solutions, or unbounded solutions whose upper bounds grow at any rate, we show that the the important asymptotic behaviour of any stable finite dimensional perturbed autonomous system 
	\[
	x'(t)=Ax(t)+f(t),\quad t\geq 0
	\] 
	is captured by $y$ above, which is in turn captured by the one--parameter family of functions $\int_{t-\theta}^t f(s)\,ds$. In fact, as we will see, the asymptotic behaviour of this family of functions is essentially the same for all values of $\theta$.
	
	In a following work, we will show that this is also the case for perturbed Volterra differential equations of the form 
	\begin{equation} \label{eq.volterra}
		x'(t)=\int_{[0,t]} \nu(ds)x(t-s)+f(t), \quad t\geq 0
	\end{equation}
	where $\nu$ is a $d\times d$ matrix--valued finite measure, and $f$ is a $\mathbb{R}^d$--valued function, provided the underlying differential resolvent $r$, which solves the matrix--valued equation 
	\[
	r'(t)=\int_{[0,t]} \nu(ds)r(t-s), \quad t\geq 0; \quad r(0)=I_{d\times d},
	\]
	lies in $L^1([0,\infty);\mathbb{R}^{d\times d})$. 
	
	Although the programme appears very humble, it turns out the question of what constitute necessary and sufficient conditions for sharp $L_\infty$ estimates of solutions of perturbed equations, seems open. As hinted at in the first section, it appears that this is a question that is worthwhile investing effort in, since a clear picture of the asymptotic behaviour of equations \eqref{eq.y1}, \eqref{eq.Y01} and \eqref{eq.Y1} would appear to be very useful in analysing the asymptotic behaviour of more complicated dynamical systems. 
	
	To give a concrete example, suppose that $X$ is the solution of the perturbed stochastic pantograph equation 
	\[
	dX(t)=\left(bX(t)+aX(qt)+f(t)\right)\,dt + \sigma(t)\,dB(t)
	\]
	where $a$ and $b$ are real constants, $q\in (0,1)$ and  $X(0)=\xi$. Let $Y(0)=0$, consider $Z=X-Y$, and set $\varphi(t)=(1+b)Y(t)+aY(qt)$ for $t\geq 0$.
	Then 
	\[
	Z'(t)=bZ(t)+aZ(qt)+\varphi(t), \quad t\geq 0.
	\]
	Thus, the asymptotic behaviour of $\varphi$ is entirely captured by that of $Y$, and studying the asymptotic behaviour of $Z$ is certainly easier than studying that of $X$ directly. Moreover, since $X=Z+Y$, \textit{$X$ depends on the perturbations $f$ and $\sigma$ purely through $Y$}.  Conversely, writing 
	\[
	dX(t)=\left(-X(t)+f(t)+\{(1+b)X(t)+aX(qt)\}\right)\,dt + \sigma(t)\,dB(t)
	\]
	and treating the term in curly brackets as a forcing term, we may apply stochastic integration by parts to get 
	\[
	X(t)=\xi e^{-t} + Y(t)
	+\int_0^t e^{-(t-s)} \{(1+b)X(s)+aX(qs)\}\,ds, \quad t\geq 0,
	\] 
	which rearranges to give an expression for $Y$ in terms of $X$:
	\[
	Y(t)=
	X(t)-\xi e^{-t} -\int_0^t e^{-(t-s)} \{(1+b)X(s)+aX(qs)\}\,ds, \quad t\geq 0.
	\]
	Therefore, if nice asymptotic behaviour in $X$ arises, we may use this identity to obtain information about the behaviour of $Y$ that is necessary to generate that behaviour. A programme of results which characterise the $L_\infty$ asymptotic behaviour of this equation in the case $b<0$, $|a|<|b|$ (which are necessary and sufficient conditions for all solutions of the equation with $f=\sigma\equiv 0$ to tend to zero), and which uses results from this work, is given in \cite{ALpanto2024}.
	
	Removing the complication of the stochastic perturbation for a moment, we see that these remarks are still valid for deterministically forced equations. Consider the asymptotic behaviour of 
	\[
	x'(t)=bx(t)+ax(qt)+f(t), \quad t\geq 0
	\]
	under the same conditions as above. Consider $z=x-y$, and set $\phi(t)= (1+b)y(t)+ay(qt)$. Then $z$ obeys 
	\[
	z'(t)=bz(t)+az(qt)+\phi(t), \quad t\geq 0.
	\]
	This is still a perturbed equation, but the pointwise behaviour of $\phi$ is likely to be better behaved than that of $f$, due to the fact that $y$, on which $\phi$ entirely depends, is a convolution with $f$ an exponential function
	\[
	y(t)=\int_0^t e^{-(t-s)}f(s)\,ds, \quad t\geq 0
	\] 
	and therefore is likely to be smoother. Once again, the functional dependence of $x=y+z$ on $f$ comes purely through $y$, and so, if $y$ has certain asymptotic behaviour, this may enable us to obtain good perturbation results by studying the equation for $z$. Conversely, since 
	\[
	y(t)=	x(t)-\xi e^{-t} -\int_0^t e^{-(t-s)} \{(1+b)x(s)+ax(qs)\}\,ds, \quad t\geq 0,
	\] 
	we may be able to capture necessary conditions on $y$ so that $x$ has appropriate asymptotic behaviour. 
	
	Finally, if the stochastically forced pantograph equation is 
	\[
	dX_0(t)=(bX_0(t)+aX_0(qt))\,dt + \sigma(t)\,dB(t),
	\]
	it makes sense to consider $Z_0=X_0-Y_0$, so $Z_0$ obeys 
	$Z_0'(t)=bZ_0(t)+aZ_0(qt)+\phi_0(t)$ for $t\geq 0$, where $\phi_0(t)=aY_0(qt)$, so once again $X_0=Z_0+Y_0$ depends on the forcing term $\sigma$ purely through $Y_0$, and conversely 
	\[
	Y_0(t)=	X_0(t)-\xi e^{-t} -\int_0^t e^{-(t-s)} \{(1+b)X_0(s)+aX_0(qs)\}\,ds, \quad t\geq 0.
	\] 
	Thus, we see that to study different types of deterministic and stochastic forcing of the underlying unperturbed equation $v'(t)=bv(t)+av(qt)$ for $t\geq 0$, it makes sense to have asymptotic results available for $y$, $Y_0$ and $Y$, according to the situation. 
	
	The reader will notice that we considered in this example a perturbation of a scalar non--autonomous delay differential equation, where the delay is unbounded. But we can study autonomous multidimensional equations with distributed delay as well. For instance, consider the (multidimensional) stochastic Volterra equation 
	\[
	dX(t)=\left(\int_{[0,t]} \nu(ds)X(t-s) + f(t)\right)\,dt+ \sigma(t)\,dB(t), \quad t\geq 0.
	\]
	By the expedient of subtracting the process $Y_i$ from the $i$--th component $X_i$, where $Y_i$ is defined by 
	\[
	dY_i(t)=\left(-Y_i(t)+f_i(t)\right)\,dt + \sum_{j=1}^m \sigma_{ij}(t)\,dB_j(t), \quad t\geq 0,
	\] 
	and letting $Y_i$ be the component of the $\mathbb{R}^d$--valued process $Y$, we will get a nice differential system for $Z=X-Y$.  
	Here it is understood that $B$ is an $m$--dimensional standard Brownian motion, $\sigma$ is a $d\times m$--matrix valued function, and $f$ is $\mathbb{R}^d$--valued. By stochastic time change arguments, each $Y_i$ obeys an equation of the form \eqref{eq.Y1}, and $Z=X-Y$ obeys the differential equation 
	\[
	Z'(t)=(\nu \ast Z)(t)+g(t), \quad t\geq 0
	\]   
	where $g(t)=(\nu\ast Y)(t)+Y(t)$. Once again, therefore, with $r$ standing for the differential resolvent, we have 
	\[
	X(t)=r(t)X(0)+Y(t)+\int_0^t r(t-s)\{(\nu\ast Y)(s)+Y(s)\}\,ds, \quad t\geq 0,
	\]
	so the dependence of $f$ and $\sigma$ on $X$ is captured completely by $Y$, and conversely, necessary conditions on $Y$ may be inferred  via the representation 
	\[
	Y(t)=X(t)-e^{-t}X(0)-\int_0^t e^{-(t-s)}\{(\nu\ast X)(s)+X(s)\}\,ds, \quad t\geq 0,
	\]
	which is established by 
	writing the Volterra equation in the form 
	\[
	dX(t)=\left(-X(t)+ f(t) +(\nu\ast X)(t)+X(t)\right)\,dt + \sigma(t)\,dB(t),
	\]
	treating the terms in $X$ in the drift as a perturbation, and writing down the corresponding variation of constants formula for $X$.
	
	In recent works on deterministic and stochastic linear Volterra differential equations, the authors have used this approach to determine necessary and sufficient conditions on the forcing terms for solutions to lie in $L^p$ spaces \cite{AL:2024Lp}. We note also that conditions of the form 
	\[
	\lim_{t\to\infty}
	\sup_{0\leq \theta\leq 1}\left|\int_t^{t+\theta} f(s)\,ds \right|=0
	\]
	have been known to be necessary and sufficient for any solution of \eqref{eq.y1} to tend to zero (see Strauss and Yorke~\cite{SY67a,SY67b}). The uniformity has been dispensed with, in the context of Volterra differential equations by Gripenberg, Londen and Staffans~\cite{GLS} by means of an interesting decomposition Lemma, which the authors re--proved in \cite{AL:2023(AppliedMathLetters)}, and which is generalised here. A corollary of their result is that  
	\[
	\lim_{t\to\infty} 	\int_t^{t+\theta}f(s)\,ds = 0, \quad \text{for each $\theta>0$}
	\]  
	is sufficient for $y(t)\to 0$ as $t\to\infty$, and indeed this condition  is also necessary as is easily confirmed through the identity
	\[
	\int_{t}^{t+\theta} f(s)\,ds = y(t+\theta)-y(t)+\int_t^{t+\theta} y(s)\,ds,
	\]
	which is obtained by integrating the differential equation and rearranging. In fact, we find the asymptotic behaviour of the family of functions  
	\[
	t\mapsto \int_{(t-\theta)^+}^t f(s)\,ds=:f_\theta(t)
	\]
	(rather than $f$ itself) is what is important for the $L^p$ stability of solutions of linear deterministic Volterra differential equations (the necessary and sufficient condition for solutions of \eqref{eq.volterra} to be in $L^p$ is that $f_\theta$ to be in $L^p$ for all $\theta>0$). The authors have also recently shown that the asymptotic behaviour of the family of functions $f_\theta$ characterises the case in which solutions of \eqref{eq.volterra} have long--run time averages \cite{AL:2024Cesaro}.  
	
	It should be pointed out that all such stability results require conditions on $f_\theta$ only, and do not appear to depend on the rate of decay of solutions of the unperturbed differential (or integrodifferential) equation, or to whether the differential system is scalar or multi--dimensional. We shall demonstrate this for finite dimensional ODEs in this paper. In part, therefore, this justifies focussing attention on the very simple differential equation \eqref{eq.y1}.  
	
	As we will see in Part II, the situation for stochastic equations (or stochastic perturbations) appears to be similar. For instance, it has been known for several years that the pathwise behaviour of $Y_0$ (namely whether it  tends to zero, is bounded but not convergent, or is unbounded, a.s.) depends on $\sigma$ entirely through the asymptotic behaviour of the sequence $\sigma_1^2$ defined by 
	\[
	n\mapsto \sigma_1^2(n):=\int_{n-1}^{n}\sigma^2(s)\,ds \]
	so once again, it is the \textit{behaviour of the integral of forcing terms on fixed moving compact intervals} that appears decisive. See \cite{ACR:2011(DCDS)}. This is also reflected in studying $L^p$ stability of solutions of Volterra equations of the form 
	\[
	dX(t)=\int_{[0,t]} \nu(ds)X(t-s) \,dt + \sigma(t)\,dB(t), \quad t\geq 0, 
	\]
	for which solutions are in $L^p$ a.s. (for $p\geq 2$, for instance) if and only if 
	\[
	t\mapsto \int_{t}^{t+1} \sigma^2(s)\,ds \in L^{p/2}(0,\infty).
	\]
	
	In this paper, in contrast to recent studies on $L^p$ spaces (for finite $p$) we will study \emph{pointwise} asymptotic behaviour of the solutions of the equations \eqref{eq.y1}, \eqref{eq.Y01} and \eqref{eq.Y1}; that is to say, we are interested in the $L_\infty$ behaviour of solutions. In the case of stochastic equations, we are interested in pathwise (or almost sure dynamics), rather than the behaviour in mean, mean square or $p$--th mean more generally, though again we note that good behaviour of forcing terms on compact intervals appears to be salient for such equations in general also, as is explored in Appleby and Lawless \cite{AL:2023(AppliedNumMath)}. 
	
	As pointed out already, the necessary and sufficient conditions for the convergence to zero and boundedness of \eqref{eq.Y01} and \eqref{eq.y1} have already been established. Here, we extend these results to deal with the doubly perturbed equation \eqref{eq.Y1}. We also give necessary and sufficient conditions for the solutions of \eqref{eq.y1}, \eqref{eq.Y01} and \eqref{eq.Y1} to lie in certain $L_\infty$ weighted spaces. In particular, we give necessary and sufficient conditions on when 
	\[
	\frac{|y(t)|}{\gamma(t)}, \quad \frac{|Y_0(t)|}{\gamma(t)}, \quad \frac{|Y(t)|}{\gamma(t)}
	\]  
	have finite and non--trivial limits superior, particularly in the case where $\gamma$ is a so--called subexponential function (that is to say, a function that grows or tends to zero in a way that is slower than any real exponential function) or a certain broad class of functions which grow no faster than exponentially. Such results are also carried out in the deterministic case when $\gamma$ is an increasing function. In all cases, we find that it is the behaviour of $f_\theta$ or $\sigma^2_1$ that decides the asymptotic behaviour. To whet the appetite, the following result is typical:
	\begin{theorem}
		Let $f$ be continuous, and $y$ be the solution of \eqref{eq.y1}. Let $\gamma$ be a continuous, positive and non--decreasing function. Then the following are equivalent:
		\begin{enumerate}
			\item[(A)] There exists $\delta>0$ such that for some $K>0$ and $\theta'\in(0,\delta]$ we have 
			\begin{gather*}
				\limsup_{t\to\infty} \frac{\left|\int_{t-\theta'}^t f(s)\,ds\right|}{\gamma(t)}>0, \\
				\left|\int_{(t-\theta)^+}^t f(s)\,ds\right| \leq K\gamma(t), \quad \theta\in [0,\delta], t\geq 0;
			\end{gather*}	
			\item[(B)] For each $\delta>0$ there exists $K=K(\delta)>0$ and $\theta'\in (0,\delta]$ such that 
			\begin{gather*}
				\limsup_{t\to\infty} \frac{\left|\int_{t-\theta'}^t f(s)\,ds\right|}{\gamma(t)}>0, \\
				\left|\int_{(t-\theta)^+}^t f(s)\,ds\right| \leq K\gamma(t), \quad \theta\in [0,\delta], t\geq 0;
			\end{gather*}
			\item[(C)] For every $\delta>0$, there exists $K=K(\delta)>0$ such that 
			\[
			\left|\int_{(t-\theta)^+}^t f(s)\,ds\right| \leq K\gamma(t), \quad \theta\in [0,\delta], t\geq 0,
			\]
			and 
			\[
			\text{Leb}\left(\theta\in [0,\delta]:
			\limsup_{t\to\infty} \frac{\left|\int_{t-\theta}^t f(s)\,ds\right|}{\gamma(t)}=0\right)=0.
			\]
			\item[(D)] There exists $\Lambda_\gamma y\in (0,\infty)$ such that 
			\[
			\Lambda_\gamma y := \limsup_{t\to\infty} \frac{|y(t)|}{\gamma(t)}.
			\] 
		\end{enumerate}	
	\end{theorem} 
	We also show how results extend naturally, and with little modification to finite dimensions, as well as establishing in the case of deterministic equations, a characterisation of the asymptotic behaviour derivative of the solution. 
	
	In light of the importance of the asymptotic behavior of the average of $f$ in studying the asymptotic behaviour of solutions, it might be questioned what role if any $f$ itself may play. It turns out that to understand this, we should look at the asymptotic behaviour of the derivative. Using the result above to orient our discussion, we find that if both the solution and the derivative are to be $O(\gamma)$, then the forcing function itself, as well as the integral average, must be $O(\gamma)$, and conversely. On the other hand, the derivative will have a faster growing upper bound than the solution if and only if the forcing function has a faster growing bound than the average. Thus, ``bad'' behaviour in the derivative, relative to the solution, is symptomatic of well--behaved average behaviour and bad pointwise behaviour in the forcing term. Radically different asymptotic behaviour of forcing terms and their averages can be observed: we give examples which demonstrate that the forcing term can have arbitarily large growth bounds, and yet the average can be bounded, or even tend to zero. This shows the potential utility and sharpness of results.  
	
	\section{Topics not examined in the paper}
	Before moving forward to general results, we pause to indicate some questions we have not addressed within this article, and our reasons for doing so.
	
	\subsection{The general scalar linear equation} 
	The reader may wonder why we do not study perturbations of the general linear equation viz.,
	\begin{equation} \label{eq.ualpha}
		u'(t)=-\alpha u(t), \quad t\geq 0
	\end{equation}
	where $\alpha>0$. In part, this is because we consider in the last section general $d$--dimensional deterministic equations which are perturbed, and recover exactly the same type bounds on solutions when the growth of $f_\theta$ is monotone, or exhibits slower than exponential growth or decay features. Note, moreover, by appealing to the theorem above, that the order of these bounds depend on $f_\theta$, but not on the strength of the mean reversion of the system, as is apparently captured by the presence of the exponential term in $y$. Moreover, we can recover the general scalar case \eqref{eq.ualpha} by setting $d=1$. However, this argument is rather indirect, so we give a direct reason. 
	
	Suppose that we perturb \eqref{eq.ualpha} deterministically, to get 
	\begin{equation} \label{eq.xalpha}
		x'(t)=-\alpha x(t)+f(t), \quad t\geq 0; \quad x(0)=\zeta.
	\end{equation}
	Consider $z(t)=x(t)-y(t)$, and set $\phi(t)=(1-\alpha)y(t)$ for $t\geq 0$. Then $z'(t)=-\alpha z(t)+(1-\alpha)y(t)$ for $t\geq 0$, so 
	\[
	x(t)=y(t)+z(t)=e^{-\alpha t}\zeta + y(t)+\int_0^t e^{-\alpha(t-s)}(1-\alpha)y(s)\,ds, \quad t\geq 0.
	\]
	Hence we see that $x$ depends on $f$ purely through $y$, and because our results suggest the asymptotic behaviour of $y$ is governed by $f$ through $f_\theta$, no fundamentally new behaviour seems to be generated in the general case.
	
	\subsection{Exponentially decaying perturbations} In our introduction, we have focussed on the case where the size of the perturbations (measured via $f_\theta$) are growing, or are dominated by exponential functions (either growing or decaying). This leaves the important class of exponentially (or faster than exponentially) decaying perturbations unexamined. Here, the situation can be approached rather directly, and it is not clear to us that consideration of $f_\theta$ leads to the best formulation of results, especially if we are not greatly concerned with capturing exactly the Liapunov exponent which characterises the rate at which solutions decay. This is however an important question, and demands attention, particularly for stochastic equations, and we will return to it in a later work. For now, we justify our elimination of exponentially decaying perturbations (or exponential decay in solutions) by giving simple characterisation of exponential stability for \eqref{eq.xalpha}. 
	
	We do not claim the result below is novel, but give a proof since (a) it is elementary and short (b) it makes the work self--contained and (c) it demonstrates that the approach is different from what follows for larger perturbations. 
	\begin{theorem}
		Let $\alpha>0$, $f$ be continuous, and suppose $x$ is the unique continuous solution to \eqref{eq.xalpha}. Then the following are equivalent
		\begin{enumerate}
			\item[(A)] The limit 
			\[
			L:=\lim_{t\to\infty} \int_0^t f(s)\,ds
			\]
			is finite, so the function $F(t) = L-\int_0^t f(s)\,ds$ is 
			well--defined and obeys $|F(t)|\leq Ce^{-\eta t}$ for all $t\geq 0$ and some constants $C>0$ and $\eta>0$; 
			\item[(B)] There exists $\beta>0$ and $K>0$ such that 
			$|x(t)|\leq Ke^{-\beta t}$ for all $t\geq 0$. 
		\end{enumerate}	
	\end{theorem}
	\begin{proof}
		Suppose (A) holds. Then $F$ is well--defined and have $F'(t)=-f(t)$ for $t\geq 0$. Thus, by variation of constants, noting that $F'(t)=-f(t)$ and then integrating by parts, we get  
		\begin{align*}
			x(t)
			&=\zeta e^{-\alpha t} - e^{-\alpha t}\int_0^t e^{\alpha s}F'(s)\,ds\\
			&=\zeta e^{-\alpha t} - F(t)+F(0)e^{-\alpha t}+\int_0^t\alpha e^{-\alpha(t-s)} F(s)\,ds
		\end{align*}
		Hence
		\[
		|x(t)|\leq (|\zeta|+|L|)e^{-\alpha t} + Ce^{-\eta t} 
		+C\alpha e^{-\alpha t}\int_0^t e^{(\alpha-\eta) s}\,ds.
		\]
		If $\alpha>\eta$, the integral is order $e^{(\alpha-\eta)t}$, so the last term on the right hand side is $O(e^{-\eta t})$, and we may take $\beta=\eta$ to get (B). If $\alpha<\eta$, the integral tends to a limit, and the last term is $O(e^{-\alpha t})$, and we may take $\beta=\alpha$ to get (B). If $\alpha=\eta$, the integral is $t$, so any choice of $\beta\in (0,\eta)$ gives (B). Hence (A) implies (B).
		
		To show (B) implies (A), integrate the equation over $[0,t]$ and rearrange to get 
		\[
		\int_0^t f(s)\,ds= 
		x(t)-\zeta +\alpha \int_0^t x(s)\,ds.  
		\] 
		Since (B) holds, all the terms on the right hand side have finite limits as $t\to\infty$, so $L:=\lim_{t\to\infty} \int_0^t f(s)\,ds$ is finite. Take $t\geq 0$ fixed and let $T>t$. Integration over $[t,T]$ and rearranging gives 
		\[
		\int_t^T f(s)\,ds = x(T)-x(t) + \alpha\int_t^T x(s)\,ds.
		\]
		The righthand side has absolute value bounded by 
		\[
		K\left(e^{-\beta T} + e^{-\beta t} + \alpha \int_t^T e^{-\beta s}\,ds\right).
		\]
		Letting $T\to\infty$, we see that left hand side tends to $F(t)$; thus, for any $t\geq 0$, by taking $T\to\infty$, we get   
		\begin{align*}
			|F(t)|&=\left|\lim_{T\to\infty} \int_t^T f(s)\,ds \right|\\
			&\leq \limsup_{T\to\infty} 	K\left(e^{-\beta T} + e^{-\beta t} + \alpha \int_t^T e^{-\beta s}\,ds\right)
			=  K\left(1 + \frac{\alpha}{\beta}\right)e^{-\beta t}.
		\end{align*}
		Thus (A) holds, with $\eta=\beta$. 
	\end{proof}
	
	\subsection{Multidimensional equations}
	Finally, the reader may wonder whether the results quoted here are only possible to prove in the scalar case, and that extensions to higher dimensions do not carry over (either readily or at all). To dispel such worries (in the deterministic case at least) we state in the final section some analogues of Theorems in Section 3 for the deterministic equation 
	\begin{equation} \label{eq.xmult0}
		x'(t)=Ax(t)+F(t), \quad t\geq 0; \quad x(0)=\zeta
	\end{equation}
	where the solution is in $\mathbb{R}^d$, $A$ is a $d\times d$ matrix all of whose eigenvalues have negative real parts, $F\in C([0,\infty);\mathbb{R}^d)$ and $\zeta\in \mathbb{R}^d$. We sketch the proof of one or two results to satisfy the interested reader that these are in fact obvious generalisations of scalar proofs. Moreover, the reader will notice that these proofs rely directly on results proven in the scalar case, especially by application of scalar results to components. This partly justifies, on non--didactic grounds, our decision to work out the details in the simpler, scalar case first.     
	
	\section{Deterministically Forced Scalar Equation}
	Introduce the ordinary differential equation 
	\begin{equation} \label{eq.y}
		y'(t)=-y(t)+f(t), \quad t\geq 0; \quad y(0)=0.
	\end{equation}  
	There is no loss in generality in taking the initial condition to be zero. Suppose for simplicity that
	\begin{equation} \label{eq.fcns}
		f\in C([0,\infty);\mathbb{R}).
	\end{equation}  
	Scrutiny of the proofs in this section reveals that $f$ being  locally in $L^1$ is in fact sufficient for the results in this section; however, we take $f$ continuous throughout to align with assumptions for stochastic equations.	
	
	Under these assumptions on $f$, \eqref{eq.y} has a unique continuous solution, given by 
	\[
	y(t)=\int_0^t e^{-(t-s)}f(s)\,ds, \quad t\geq 0. 
	\]
	\subsection{A decomposition lemma}
	The following result is known, but we will discuss it in order to motivate the direction of our results.  
	\begin{theorem} \label{thm.detasymptoticstab} 
		Let $f\in C([0,\infty);\mathbb{R})$. Then the following are equivalent
		\begin{enumerate}
			\item[(A)] $y(t)\to 0$ as $t\to\infty$;
			\item[(B)] For every $\theta>0$, $\lim_{t\to\infty} \int_t^{t+\theta} f(s)\,ds =0$; 
		\end{enumerate}
	\end{theorem}
	The fact that (B) implies (A) is due to \cite[Lemma 15.9.2]{GLS}  but the contribution in Strauss and Yorke \cite{SY67a,SY67b} established the importance of a condition of this type; Strauss--Yorke require uniform convergence in $\theta$. As pointed out in \cite{AL:2023(AppliedNumMath)}, it is easy to see that (A) implies (B) by integrating across \eqref{eq.y} and re--arranging: 
	\begin{equation}\label{eq.aveRepy}
		\int_{t-\theta}^t f(s)\,ds = y(t) - y(t-\theta) +\int_{t-\theta}^t y(s)\,ds.
	\end{equation}
	This identity is important, and we wish to highlight another crucial set of identities connecting the average of $f$ and $y$. A variant of it appears in \cite{GLS} and also in \cite{AL:2023(AppliedMathLetters)} (in these works, the parameter $\Delta$ below is equal to unity).
	\begin{lemma} \label{lemma.yintermsofftheta} Let $\Delta>0$. 
		Let $f=0$ for $t<0$. Set $f_\theta(t)=\int_{t-\theta}^t f(s)\,ds$ for $t\geq 0$. Define $\delta(t)=f(t)-f_\Delta(t)/\Delta$ for $t\geq 0$. Then with 
		\begin{equation} \label{eq.I}
			I(t)=\int_0^t \delta(s)\,ds, \quad t\geq 0,
		\end{equation}
		we have
		\begin{eqnarray} \label{eq.intdelta}
			I(t)&=& \frac{1}{\Delta}\int_0^\Delta f_\theta(t) \,d\theta, \quad t\geq \Delta,\\
			I(t)&=& \int_0^t f(s)\,ds -\frac{1}{\Delta}\int_0^t \int_0^s f(u)\,du\,ds
			=f_\Delta(t)-\frac{1}{\Delta}\int_0^t f_\Delta(s)\,ds, \quad t\in [0,\Delta], \nonumber
		\end{eqnarray}
		and if $y$ obeys \eqref{eq.y}, then 
		\begin{equation} \label{eq.yrepave}
			y(t)=\frac{1}{\Delta}\int_0^t e^{-(t-s)}f_\Delta(s)\,ds + I(t) - \int_0^t e^{-(t-s)} I(s)\,ds.
		\end{equation}
	\end{lemma} 
	The proof is a trivial re--working of the argument given in \cite{AL:2023(AppliedMathLetters)}, so we omit it. 
	
	Notice that \eqref{eq.yrepave} implies the following important fact: \textit{for any $\Delta>0$}, $y$ depends on $f$ \textit{only through} the family of functions $\{f_\theta:\theta\in [0,\Delta]\}$. This should lead the reader to the (correct) conclusion that if $f_\theta$ has the appropriate behaviour for all $\theta$ in an interval, \textit{however short that interval is}, this behaviour will be inherited by $y$. This claim is guided by the intuition from admissibility theory that integrals (such as $I$) and convolutions preserve nice behaviour from $f_\theta$, and that such integrals and convolutions furnish the representation of $y$ in \eqref{eq.yrepave}. On the other hand, \eqref{eq.aveRepy} suggests that $f_\theta$ can be written purely in terms of $y$, so appropriate behaviour in $y$ can only arise if $f_\theta$ possesses that behaviour, for more or less any choice of $\theta$. A good deal of our analysis involves making these intuitions precise.  
	
	\subsection{Upper bounds}
	We use \eqref{eq.aveRepy}, \eqref{eq.intdelta} and \eqref{eq.yrepave} extensively. Here is an application to asymptotic behaviour when we desire the differential equation \eqref{eq.y} to have a non--decreasing upper bound.
	\begin{theorem} \label{lemma.fthetaymonotone}
		Let $f\in C([0,\infty);\mathbb{R})$, and $y$ be the unique continuous solution to \eqref{eq.y}. 
		Let $\gamma$ be positive, non--decreasing and in $C(0,\infty)$. Then the following are equivalent:
		\begin{enumerate}
			\item[(A)] There exists $\Delta>0$ and $C=C(\Delta)>0$ such that $|\int_{(t-\theta)^+}^t f(s)\,ds|\leq C\gamma(t)$ for all $\theta\in [0,\Delta]$;
			\item[(B)] For all $\Delta>0$, there is a $C=C(\Delta)>0$ such that $|\int_{(t-\theta)^+}^t f(s)\,ds|\leq C\gamma(t)$ for all $\theta\in [0,\Delta]$;
			\item[(C)] $y(t)=\text{O}(\gamma(t))$, as $t\to\infty$;	
		\end{enumerate}
	\end{theorem}
	\begin{proof}
		Let $\Delta>0$ be arbitrary. We show that (C) implies (B).
		By (C), $|y(t)|\leq K \gamma(t)$ for some $K>0$. Apply the triangle inequality to \eqref{eq.aveRepy}, and using the monotonicity of $\gamma$, we get for all $t\geq \theta$ 
		\[
		\left|
		\int_{t-\theta}^t f(s)\,ds\right| \leq K\gamma(t) +K\gamma(t-\theta) +\int_{t-\theta}^t K\gamma(s)\,ds\leq 2K\gamma(t)+\theta\gamma(t)\leq (2K+K\Delta)\gamma(t),
		\]  
		which is (B) with $C=K(2+\Delta)$. For $t\in [0,\theta]$ 
		\[
		\int_{(t-\theta)^+}^t f(s)\,ds=\int_0^t f(s)\,ds = y(t)+\int_0^t y(s)\,ds,
		\]
		The righthand side has absolute value less that $K\gamma(t)+tK\gamma(t)$, which in turn is less than $K(1+\theta)\gamma(t)\leq K(1+\Delta)\gamma(t)$, since $\theta\leq \Delta$. Hence (B) holds once more with $C=K(2+\Delta)$.   
		
		Clearly (B) implies (A), so it remains to show that (A) implies (C) to complete the chain of implications. In this part of the proof $\Delta>0$, rather than being arbitrary, is merely a constant for which (A) holds. By this, we have $|f_\theta(t)|\leq K\gamma(t)$ for all $t\geq 0$, and $\theta\in [0,\Delta]$. For $t\geq \Delta$, 
		we therefore have by \eqref{eq.intdelta} the estimate 
		\[
		|I(t)|\leq \frac{1}{\Delta}\int_0^\Delta |f_\theta(t)|\,d\theta\leq K\gamma(t).
		\]
		For $t\in [0,\Delta]$, we have $|f_\Delta(t)|\leq K\gamma(t)$, and by the equation after \eqref{eq.intdelta}, we get 
		\[
		|I(t)|\leq |f_\Delta(t)| + \frac{1}{\Delta}\int_0^t |f_\Delta(s)| \,ds \leq K\gamma(t)+\frac{1}{\Delta}\int_0^t K\gamma(s)\,ds \leq 2K\gamma(t).
		\]
		Thus, $|I(t)|\leq 2K\gamma(t)$ for all $t\geq 0$. Now, take the triangle inequality in \eqref{eq.yrepave}, apply the above estimates for $I$ and $f_1$, and use the monotonicity of $\gamma$ to get  
		\begin{align*} 
			|y(t)|&\leq \frac{1}{\Delta}\int_0^t e^{-(t-s)}|f_\Delta(s)|\,ds + |I(t)| + \int_0^t e^{-(t-s)} |I(s)|\,ds\\
			&\leq \frac{1}{\Delta}\int_0^t e^{-(t-s)}K\gamma(s) + 2K\gamma(t) + \int_0^t e^{-(t-s)} 2K\gamma(s)\,ds\\
			&\leq \frac{K}{\Delta}\gamma(t)+4K\gamma(t)=K(4+1/\Delta)\gamma(t),
		\end{align*}
		as needed. 
	\end{proof}
	
	It is easy to use admissibility results due to Corduneanu \cite{cor1,cor2,cor3} to show that whenever the eigenvalues of $A$ have negative real parts, $\gamma$ is increasing and $|f(t)|\leq C\gamma(t)$ for all $t\geq 0$, then the solution of the general finite--dimensional linear equation 
	\[
	x'(t)=Ax(t)+f(t)
	\]
	obeys $|x(t)|\leq K'\gamma(t)$, for some $K'>0$, which is the type of growth bound we have found above. However, as we have seen, such a pointwise condition on $f$ is not necessary to give solutions bounded by $\gamma$, although it is clearly sufficient (and indeed implies the bound in part (A) of the above theorem, with $K=C$). 
	
	A direct application of the admissibility approach should be approached cautiously, since if $f$ has very large pointwise bounds this may leading to an overly--conservative estimate of the solution, while $f_\theta$ may be growing much more slowly. 
	
	To see this, let $\beta\in C^2((0,\infty);\mathbb{R})$ be strictly increasing with $\beta(t), \beta'(t) \to \infty$ as $t\to\infty$. Let $f : \mathbb{R}^\to \mathbb{R}$ be given by $f(t)= \beta'(t) \sin(\beta(t))$ for $t\geq 0$. We have the pointwise bound $|f(t)|\leq \beta'(t)$ for $t\geq 0$; moreover this bound is sharp, since it is achieved at all times $t$ for which $\sin(\beta(t))=\pm1$. Thus, by considering a very rapidly growing $\beta$ and naively applying the admissibility theory, we get a very rapidly growing upper estimate on the rate of growth of solutions of \eqref{eq.y}.    
	
	The very rapid oscillation of the function gives rise to radically different (and milder behaviour in $f_\theta$. Integration by substitution gives 
	\[
	\int_1^t f(s)\,ds =-\cos(\beta(t))+\cos(1)
	\]
	Thus 
	\[
	\int_{t-\theta}^t f(s)\,ds  
	=-\cos(\beta(t)) +\cos(\beta(t-\theta)),
	\]
	so $|f_\theta(t)|\leq 2$ for all $\theta\in [0,1]$ and $t\geq 0$, but $f_\theta(t)$ does not tend to zero as $t\to\infty$. 
	
	Now we notice if we take $\gamma(t)=1$ for all $t\geq 0$, we have a characterisation of boundedness of solutions (we state the result momentarily). Later, we will be able to characterise the situation where $y(t)$ is bounded, but not convergent to zero. 
	
	\begin{theorem} \label{lemma.fthetaymonotonebdd}
		Let $f\in C([0,\infty);\mathbb{R})$, and $y$ be the unique continuous solution to \eqref{eq.y}.  Then the following are equivalent:
		\begin{enumerate}
			\item[(A)] There exists $\Delta>0$ and $C>0$ such that $|\int_{(t-\theta)^+}^t f(s)\,ds|\leq C$ for all $\theta\in [0,\Delta]$ and $t\geq 0$; 
			\item[(B)] For every $\Delta>0$ there is a $C=C(\Delta)>0$ such that $|\int_{(t-\theta)^+}^t f(s)\,ds|\leq C$ for all $\theta\in [0,\Delta]$ and $t\geq 0$;
			\item[(C)] There is $D>0$ such that $|y(t)|\leq D$ for all $t\geq 0$.	
		\end{enumerate}
	\end{theorem}
	%
	
	Next, we show that if the bound on $y$ is subexponential (which can allow for decay of $y$ to zero), this can only arise from subexponentially bounded $f_\theta$. To be precise, we give a definition of a subexponential function, which varies slightly from others. To avoid confusion in terminology in other works, we call such subexponential functions ``finite lag subexponential'' functions. 
	
	\begin{definition} A function $\gamma\in C([0,\infty);(0,\infty))$ is a \textit{finite lag subexponential} function if 
		\begin{equation} \label{eq.psisub}
			\lim_{t\to\infty} \frac{\gamma(t-\theta)}{\gamma(t)}=1, \quad\theta>0.
		\end{equation}
	\end{definition}
	
	Throughout this paper, we will call such functions ``subexponential functions'' freely, and any functions called subexponential will mean a function obeying this definition. The condition \eqref{eq.psisub} was introduced for real--valued functions in \cite{AppRey02}, where the class of subexponential functions also possess a property concerning their convolutions, and are restricted to be (absolutely) integrable, developing ideas of subexponential distributions introduced in 
	\cite{Choveretal}. 
	
	We make some brief comments, and prove some important facts, about this class of functions; some of these remarks are made in \cite{AppRey02} also. 
	
	First, the convergence in \eqref{eq.psisub} is in fact uniform on $\theta$--compacts. To see this, define the function $\rho:[0,\infty)\to (0,\infty)$ by 
	\[
	\rho(t)=\gamma(\log t), \quad t\geq 1,
	\]
	with $\rho(t)=\gamma(0)$ for $t\in [0,1]$. Then $\rho$ is positive and continuous, and, for any $\lambda>0$, obeys  
	\[
	\frac{\rho(\lambda t)}{\rho(t)}=\frac{\gamma(\log t+\log \lambda)}{\gamma(\log t)}\to 1, \quad t\to\infty,
	\]
	using \eqref{eq.psisub}. Hence $\rho$ is a slowly varying function at infinity (see \cite{BGT}). By the uniform convergence theorem for slowly varying functions (see e.g., \cite{BGT}), this means, for any $0<\Lambda<1$, that we have 
	\[
	\lim_{t\to\infty} \sup_{\lambda\in [\Lambda,1]} \left|\frac{\rho(\lambda t)}{\rho(t)}-1\right|=0.
	\]
	Hence, putting $\theta:=-\log \lambda\in [-\log \Lambda,0]$, $\Theta:=-\log\Lambda>1$, we get 
	\begin{equation} \label{eq.subexpuniform}
		\lim_{t\to\infty} \sup_{\theta\in [0,\Theta]} \left|\frac{\gamma(t- \theta)}{\gamma(t)}-1\right|=0,
	\end{equation}
	and since the choice of $\Lambda\in (0,1)$ is arbitrary, so is the choice of $\Theta>0$. \eqref{eq.subexpuniform} gives the desired uniform convergence. 
	
	It is not hard to show that one can equally consider ``advanced'' intervals, as well as ``delayed'' ones. From \eqref{eq.psisub}, it readily follows for each $\theta>0$ that 
	\[
	\lim_{t\to\infty} \frac{\gamma(t+\theta)}{\gamma(t)}=1,
	\]
	and retracing the argument above, we get for any $\Theta>0$ that 
	\[
	\lim_{t\to\infty} \sup_{\theta\in [0,\Theta]} \left| \frac{\gamma(t+\theta)}{\gamma(t)}-1\right|=0.
	\]
	
	Second, if we take the function $\Gamma(t)=\int_{t}^{t+1} \gamma(s)\,ds$, we see that $\Gamma$ is in $C^1((0,\infty);(0,\infty))$, and by the last limit and the identity
	\[
	\frac{\Gamma(t)}{\gamma(t)}-1
	=\int_0^1 \left\{\frac{\gamma(t+\theta)}{\gamma(t)}-1\right\}\,d\theta,
	\] 
	we have that $\Gamma(t)\sim \gamma(t)$ as $t\to\infty$. Note therefore that $\Gamma(t+1)/\Gamma(t)\to 1$ as $t\to\infty$, and thus we have 
	\[
	\frac{\Gamma'(t)}{\Gamma(t)} = \frac{\gamma(t+1)}{\Gamma(t+1)}\cdot \frac{\Gamma(t+1)}{\Gamma(t)}-\frac{\gamma(t)}{\Gamma(t)}\to 1-1=0,
	\] 
	as $t\to\infty$. Therefore, it is easy to manufacture subexponential functions with the same asymptotic properties, but with additional control and regularity features. In particular, we have just shown that for every subexponential function $\gamma$ there exists a $C^1$ subexponential function $\Gamma$ such that $\Gamma(t)\sim \gamma(t)$ as $t\to\infty$, and $\Gamma'(t)/\Gamma(t)\to 0$ as $t\to\infty$. 
	
	Third, we notice that asymptotic integration of $\Gamma'(t)/\Gamma(t)\to 0$ as $t\to\infty$ leads to $\log \Gamma(t)/t\to 0$ as $t\to\infty$, and in turn to $\log \gamma(t)/t\to 0$ as $t\to\infty$. This implies for every $\epsilon>0$ that 
	\[
	\lim_{t\to\infty} e^{\epsilon t} \gamma(t)=+\infty, \quad 
	\lim_{t\to\infty} e^{-\epsilon t} \gamma(t)=0,
	\]
	so $\gamma$ asymptotically dominates every negative exponential function, but is dominated by every positive exponential function. This fact justifies coining the term subexponential for this class of functions.  
	
	Fourth, a consequence of the second remark is that for any $\alpha>0$ we have 
	\begin{equation} \label{eq.subconvexp}
		\int_0^t e^{-\alpha(t-s)}\gamma(s)\,ds \sim \frac{1}{\alpha}\gamma(t), \quad t\to \infty
	\end{equation}
	which can be proven by applying L'H\^opital's rule to first factor on the righthand side of   
	\[
	\frac{\int_0^t e^{\alpha s}\gamma(s)\,ds}{e^{\alpha t} \gamma(t)}
	=\frac{\int_0^t e^{\alpha s}\gamma(s)\,ds}{e^{\alpha t} \Gamma(t)} \cdot \frac{\Gamma(t)}{\gamma(t)},
	\] 
	which, by virtue of the fact that $\gamma$ is subexponential, is an indeterminate limit of the form $+\infty/+\infty$. Specifically, we have
	\[
	\lim_{t\to\infty} \frac{\int_0^t e^{\alpha s}\gamma(s)\,ds}{e^{\alpha t} \Gamma(t)} 
	=
	\lim_{t\to\infty}
	\frac{e^{\alpha t}\gamma(t)}{\alpha e^{\alpha t} \Gamma(t)+e^{\alpha t}\Gamma'(t)} 
	=\frac{1}{\alpha},
	\] 
	using the facts that $\Gamma'(t)=\text{o}(\Gamma(t))$ and $\Gamma(t)\sim \gamma(t)$ as $t\to\infty$.
	
	We are using subexponential functions as weight functions (so that $\gamma$ is positive on $[0,\infty)$), but the class of functions that are asymptotic to a subexponential weight is a vector space. Functions of the form $\gamma_1(t)\sim At^\alpha$, and $\gamma_2(t)\sim A(\log t)^\alpha$, $\gamma_3(t) \sim At^\alpha\log t^\beta\log\log t^\delta$ for $\alpha$, $\beta$, $\delta\neq 0$ (and $A\neq 0$) are examples of slowly growing or decaying functions asymptotic to a subexponential weight. In general, the class of regularly varying functions at infinity (see \cite{BGT}) are subexponential in this sense. Functions which grow to infinity or decay to zero faster than a power law, but still slower than exponentially, are also subexponential. Examples include $\gamma_4(t)\sim A e^{Bt^\alpha}$ where $A$ and $B$ are non--zero, and $\alpha\in (0,1)$. 
	
	It should also be pointed out that a very simple subexponential function is $\gamma(t)=1$ for all $t\geq 0$. This means that many ``unweighted'' results can be read off as simple corollaries of our general theorems by taking the weight function $\gamma$ to be the unit constant function. 
	
	As well as describing a wide range of growth and decay types, subexponential weights are very robust to many operations. 
	For instance the sum, the product, the scalar product, and the maximum of such subexponential weights are themselves subexponential weights; on the other hand, if $\gamma$ is a finite lag subexponential function with 
	$\int_0^t \gamma(s)\,ds\to\infty$ as $t\to\infty$, then so is 
	\[
	\Gamma_1(t):=\int_0^t \gamma(s)\,ds
	\]    
	while if $\gamma$ is in $L^1(0,\infty)$, the (well--defined) function
	\[
	\Gamma_2(t):=\int_t^\infty \gamma(s)\,ds
	\] 
	is a subexponential weight, so subexponential weights are robust to infinite operations, such as integration. We have already tacitly pointed out that, for any $\delta>0$, the function
	\[
	\Gamma_\delta(t):=\frac{1}{\delta}\int_t^{t+\delta} \gamma(s)\,ds
	\]
	is also subexponential and is such that $\Gamma_\delta(t)\sim \gamma(t)$ as $t\to\infty$. 
	
	We have all the ingredients to prove the next result. 
	\begin{theorem} \label{lemma.fthetaysubexponential}
		Let $f\in C([0,\infty);\mathbb{R})$, and $y$ be the unique continuous solution to \eqref{eq.y}. 
		Let $\gamma$ be (finite lag) subexponential. Then the following are equivalent:
		\begin{enumerate}
			\item[(A)] There is $\Delta>0$ and $C=C(\Delta)>0$ such that $|\int_{(t-\theta)^+}^t f(s)\,ds|\leq C\gamma(t)$ for all $\theta\in [0,\Delta]$;
			\item[(B)] For every $\Delta>0$ there is a $C=C(\Delta)>0$ such that $|\int_{(t-\theta)^+}^t f(s)\,ds|\leq C\gamma(t)$ for all $\theta\in [0,\Delta]$;
			\item[(C)] $y(t)=\text{O}(\gamma(t))$, as $t\to\infty$;	
		\end{enumerate}
	\end{theorem}
	\begin{proof}
		Assume (C); we prove (B), which obviously implies (A). By (C), $|y(t)|\leq K\gamma(t)$ for all $t\geq 0$. Let $\Delta>0$ be arbitrary and $\theta\in [0,\Delta]$. For $t\geq \Delta\geq \theta$, by applying the triangle inequality to \eqref{eq.aveRepy}, we have 
		\[
		\frac{1}{\gamma(t)}\left|\int_{(t-\theta)^+}^t f(s)\,ds\right| \leq K+K\frac{\gamma(t-\theta)}{\gamma(t)}+K\int_{u=0}^\theta \frac{\gamma(t-u)}{\gamma(t)}\,du.
		\] 
		By the uniform convergence in the limit, there is a fixed $T>\Delta$ such that for all $u\in [0,\Delta]$, $\gamma(t-u)/\psi(t)\leq 2$ for $t\geq T$. Thus, for $t\geq T$, the righthand side is bounded by $(3+2\Delta)K$. Also, on the compact $(t,\theta)\in [\Delta,T]\times [0,\Delta]$ the continuous function 
		\[
		(t,\theta)\mapsto \frac{1}{\gamma(t)}\left|\int_{(t-\theta)^+}^t f(s)\,ds\right| =: F(t;\theta)
		\]  
		is uniformly bounded. Hence there is a $t,\theta$--independent positive constant $K'$ such that 
		\[
		\left|\int_{(t-\theta)^+}^t f(s)\,ds\right| \leq K'\gamma(t), \quad \theta\in [0,\Delta], \quad t\geq \Delta.
		\]
		To get a uniform estimate on the compact $[0,\Delta]\times [0,\Delta]$, simply appeal to the continuity of the function $F$ once again. This proves (B). 
		
		We have shown (C) implies (B) implies (A). It remains to show that (A) implies (C). To do this, we note by assumption we have $|f_\theta(t)|\leq K\gamma(t)$ for all $t\geq 0$ and $\theta\in [0,\Delta]$. Thus, by \eqref{eq.intdelta}, we have $|I(t)|\leq K\gamma(t)$ for all $t\geq \Delta$. For $t\in [0,\Delta]$, by the equation after \eqref{eq.intdelta} we have 
		\[
		|I(t)|\leq K\gamma(t)+K\int_0^t \gamma(s)\,ds \leq K'\gamma(t),
		\]
		for some $K'>K$, using the positivity and continuity of $\gamma$ to bound the integral. Thus, we have the overall estimate $|I(t)|\leq K'\gamma(t)$ for $t\geq 0$. Now, take the triangle inequality in \eqref{eq.yrepave}, scale by $\gamma$, and apply the above estimates for $I$ and $f_1$ to get 
		\begin{align*} 
			\frac{|y(t)|}{\gamma(t)}&\leq \frac{1}{\Delta}\frac{1}{\gamma(t)}\int_0^t e^{-(t-s)}|f_\Delta(s)|\,ds + \frac{|I(t)|}{\gamma(t)} + \frac{1}{\gamma(t)}\int_0^t e^{-(t-s)} |I(s)|\,ds\\
			&\leq \left(\frac{K}{\Delta}+K'\right)\frac{1}{\gamma(t)}\int_0^t e^{-(t-s)}\gamma(s)\,ds + K'. 
		\end{align*}  
		By \eqref{eq.subconvexp} with $\alpha=1$, the second factor in the first term on the right tends to 1 as $t\to\infty$, so by continuity is bounded (by $B$ say). Hence 
		$|y(t)|\leq ((K/\Delta+K')B+K')\gamma(t)$ for all $t\geq 0$, proving (B).
	\end{proof}
	
	\subsection{Lower bounds}
	The results we have obtained so far place upper bounds on the size of solutions. In order to get good lower bounds, we first characterise the situation when the solution is dominated by a weight function at infinity. We do this first in the case of subexponential weights. 
	\begin{theorem} \label{lemma.fthetaysubexponentiallittleo}
		Let $f\in C([0,\infty);\mathbb{R})$, and $y$ be the unique continuous solution to \eqref{eq.y}. 
		Let $\gamma$ be (finite lag) subexponential. Then the following are equivalent:
		\begin{enumerate}
			\item[(A)] There exists $\Delta>0$ such that $\int_{t-\theta}^t f(s)\,ds/\gamma(t)\to 0$ for all $\theta\in [0,\Delta]$;
			\item[(B)] For every $\Delta>0$, $\int_{t-\theta}^t f(s)\,ds/\gamma(t)\to 0$ for all $\theta\in [0,\Delta]$;
			\item[(C)] $y(t)=\text{o}(\gamma(t))$, as $t\to\infty$;	
		\end{enumerate}
	\end{theorem} 
	\begin{proof}
		To show that (C) implies (B), let $\Delta>0$ be arbitrary. Note by hypothesis that for every $\epsilon>0$ that $|y(t)|\leq \epsilon \gamma(t)$ for all $t\geq T_1(\epsilon)$, which we can take greater than unity, without loss of generality. Thus, for all $t\geq T_1(\epsilon)+\Delta$, and $\theta\in [0,\Delta]$, we have 
		\[
		\left|\int_{t-\theta}^t f(s)\,ds\right| \leq \epsilon \gamma(t)+\epsilon \gamma(t-\theta) +\epsilon \int_{u=0}^{\theta} \gamma(t-u)\,du
		\] 
		Next, by the uniform convergence property of $\gamma$, there exists a constant $T>\Delta$ such that $\gamma(t-u)\leq 2\gamma(t)$ for all $t\geq T$ and all $u\in [0,\Delta]$. Take $T_2(\epsilon)=\max(T_1+\Delta,T)$ and thus for $t\geq T_2(\epsilon)$ we have 
		\[
		\left|\int_{t-\theta}^t f(s)\,ds\right| \leq \epsilon \gamma(t)+2\epsilon \gamma(t) +\epsilon \int_{u=0}^{\Delta} 2\gamma(t)\,du=(3+2\Delta)\epsilon \gamma(t),
		\]
		and since $\epsilon>0$ is arbitrary, we have (B). 
		
		Clearly, (B) implies (A), so to complete the chain of equivalences, it remains to show that (A) implies (C). In this part of the proof, $\Delta>0$ is fixed. 
		To do this, by virtue of the subexponential property of $\gamma$, we note that we have $\int_t^{t+\theta} f(s)\,ds = o(\gamma(t))$ as $t\to\infty$ for all $\theta\in [0,\Delta]$. Hence we may define 
		\[
		Q_m=\left\{\theta\in [0,\Delta]: \left|\int_{t}^{t+\theta} f(s)\,ds \right|\leq \gamma(t)\right\}.
		\]  
		By continuity of the functions, $Q_m$ is closed and measurable, with $Q_{m+1}\subseteq Q_m$. Also $\cup_{m=1}^\infty Q_m=[0,\Delta]$. Therefore, we have that there exists an $m=m'$ such that $\text{Leb}(Q_m)>0$. Fix this $m$. Using the notation  
		\[
		Q_m-Q_m:=\{\theta-\theta':\theta,\theta'\in Q_m\}
		\]
		by Lemma 15.9.2 in \cite{GLS}, $Q_m-Q_m$ contains an interval $(-\delta,\delta)$ for some $\delta>0$. Let $\theta,\theta'\in [0,1]$ be in $Q_m$, and let $\theta>\theta'$ without loss. Then for $t\geq m$, we have 
		\[
		\int_{t}^{t+\theta-\theta'} f(s)\,ds = \int_t^{t+\theta} f(s)\,ds - \int_{t+\theta-\theta'}^{t+\theta} f(s)\,ds.
		\] 
		Hence
		\[
		\left|	\int_{t}^{t+\theta-\theta'} f(s)\,ds\right|\leq \gamma(t)+\gamma(t+\theta-\theta').
		\]
		Now, for all $\vartheta\in [0,\Delta]$, there is a $T>0$ such that $\gamma(t+\vartheta)\leq 2\gamma(t)$. Hence, with $T'=\max(T,m)$, we have
		\[
		\left|	\int_{t}^{t+\theta-\theta'} f(s)\,ds\right|\leq 3\gamma(t), \quad t\geq T'.
		\]
		Hence
		\[
		\left|	\int_{t}^{t+\vartheta} f(s)\,ds\right|\leq 3\gamma(t), \quad \vartheta\in (0,\delta), \quad t\geq T.
		\]
		From this, and the uniform convergence property of $\gamma$, we get the estimate 
		\[
		|f_\theta(t)|/\gamma(t) \leq K, \quad t\geq T^\ast, \quad \theta\in [0,\Delta],
		\]
		for some $K>0$ which is $t$-- and $\theta$--independent. 
		
		Now, for $t\geq \Delta$, we have $I(t)/\gamma(t)=\Delta^{-1}\int_0^\Delta f_\theta(t)/\gamma(t)\,d\theta$. Since $f_\theta(t)/\gamma(t)\to 0$ as $t\to\infty$ for each $\theta\in [0,\Delta]$, by dominated convergence, it follows that $I(t)/\gamma(t)\to 0$ as $t\to\infty$. Since $f_\Delta(t)/\gamma(t)\to 0$ as $t\to\infty$ by hypothesis, it follows from \eqref{eq.yrepave} that $y(t)=\text{o}(\gamma(t))$ as $t\to\infty$ as required in (B), provided 
		\begin{equation} \label{eq.littleosubexpconv}
			g(t)=\text{o}(\gamma(t)), \quad  t\to\infty\quad \text{ implies }
			\int_0^t e^{-(t-s)}\gamma(s)\,ds = \text{o}(\gamma(t)), \quad t\to\infty,
		\end{equation}
		where we take $g$ to be $f_\Delta$ or $I$. 
		To prove \eqref{eq.littleosubexpconv}, by hypothesis for every $\epsilon>0$ there is a $T(\epsilon)>0$ such that $|g(t)|\leq \epsilon\gamma(t)$ for $t\geq T(\epsilon)$. Then for $t\geq T(\epsilon)$, we have 
		\[
		\frac{1}{\gamma(t)e^t}\left|
		\int_0^t e^s g(s)\,ds\right| \leq \frac{1}{\gamma(t)e^t}\left|\int_0^T e^s g(s)\,ds\right|+\epsilon \frac{1}{\gamma(t)e^t}\int_T^t e^s\gamma(s)\,ds 
		\]
		Since $\gamma$ is subexponential, $\gamma(t)e^t \to \infty$ as $t\to\infty$, and using \eqref{eq.subconvexp} with 	$\alpha=1$ shows that the second term has unit limit as $t\to\infty$. Therefore  
		\[
		\limsup_{t\to\infty} \frac{1}{\gamma(t)e^t}\left|
		\int_0^t e^s g(s)\,ds\right| \leq \epsilon,
		\]
		and since $\epsilon>0$ is arbitrary, we have established  \eqref{eq.littleosubexpconv} as required. 
	\end{proof}
	
	We deliver on the promise made before the proof of this result, namely, that taken together, the last two lemmas can be used to characterise the (subexponential) size of solutions of \eqref{eq.y} exactly to leading order. Roughly, in order for $y$ to be $\text{O}(\gamma)$ but not $o(\gamma)$ it is necessary and sufficient for $f_\theta$ to be $\text{O}(\gamma)$ for all $\theta$, but not to be $o(\gamma)$ for some $\theta\in [0,\Delta]$. 
	
	\begin{theorem} \label{lemma.fthetaysubexponentialexact}
		Let $f\in C([0,\infty);\mathbb{R})$, and $y$ be the unique continuous solution to \eqref{eq.y}. 
		Let $\gamma$ be (finite lag) subexponential. Then the following are equivalent:
		\begin{enumerate}
			\item[(A)] There exists a $\Delta>0$ such that there is a $\theta'\in (0,\Delta]$ and a $C>0$ such that 
			\[
			\limsup_{t\to\infty} \frac{|\int_{t-\theta'}^t f(s)\,ds|}{\gamma(t)}>0,
			\]
			and $\left|\int_{(t-\theta)^+}^t f(s)\,ds \right|\leq K \gamma(t)$ for all $\theta\in [0,\Delta]$, $t\geq 0$;
			\item[(B)] For every $\Delta>0$ there is $\theta'=\theta'(\Delta)\in [0,\Delta]$ and a $K=K(\Delta)>0$ such that 
			\[
			\limsup_{t\to\infty} \frac{|\int_{t-\theta'}^t f(s)\,ds|}{\gamma(t)}>0,
			\]
			and $\left|\int_{(t-\theta)^+}^t f(s)\,ds \right|\leq K \gamma(t)$ for all $\theta\in [0,\Delta]$, $t\geq 0$;
			\item[(C)] 
			\[
			\limsup_{t\to\infty} \frac{|y(t)|}{\gamma(t)} \in (0,\infty).
			\]
		\end{enumerate}
	\end{theorem}
	\begin{proof}
		Assume (C); we prove (B). Then $|y(t)|\leq L\gamma(t)$ for all $t\geq 0$, so by Theorem~\ref{lemma.fthetaysubexponential}, we have the second part of (B). Suppose, by way of contradiction, that the first part of (B) is false, so that 
		\begin{equation} \label{eq.contra1}
			\limsup_{t\to\infty} \frac{|\int_{t-\theta'}^t f(s)\,ds|}{\gamma(t)}=0, \text{ for all $\theta'\in [0,\Delta]$.}
		\end{equation}
		But then, by Theorem~\ref{lemma.fthetaysubexponentiallittleo}, we must have that $y(t)=\text{o}(\gamma(t))$ as $t\to\infty$, contradicting the hypothesis (B). Therefore, the first part of (B) must also be true, and we have that (C) implies (B), and clearly (B) implies (A).
		
		To complete the proof, we must show that (A) implies (C). Let $\Delta>0$ be the lag for which (A) holds. By the second part of (A), we have $y(t)=\text{O}(\gamma(t))$ as $t\to\infty$, by Theorem~\ref{lemma.fthetaysubexponential}, so the limsup in (C) is finite. Suppose, by way of contradiction, that the limsup in (C) is zero, so that $y(t)=\text{o}(\gamma(t))$ as $t\to\infty$. But then, by the equivalence of statements (B) and (C) in  Theorem~\ref{lemma.fthetaysubexponentiallittleo}, we have that \eqref{eq.contra1} holds for our choice of $\Delta$. But this contradicts the first part of (A), which is assumed true by hypothesis. Hence, the supposition that the limsup in (C) is zero is false, and therefore the limsup in (A) is finite and positive, proving (A). This secures the desired equivalence. 
	\end{proof}
	
	It is reasonable to ask if in the case that $\gamma$ is monotone whether a result of the above type can be shown. The answer is positive. The first part of  Theorem~\ref{lemma.fthetaysubexponentiallittleo} goes through more or less verbatim. With cosmetic changes to the beginning of second part of Theorem~\ref{lemma.fthetaysubexponentiallittleo}, it can be shown that when 
	\[
	\int_{(t-\theta)^+}^t f(s)\,ds = \text{o}(\gamma(t)), \quad t\to\infty,
	\]
	then $|f_\theta(t)|\leq C\gamma(t)$ for all $t\geq T$ and all $\theta\in [0,\delta]$ for some $\delta\in (0,\Delta)$, where $C$ is $\theta$-- and $t$--independent. We show how this bound can be extended to all $\theta\in [0,\Delta]$, since the proof was glossed over above. To start, note we have a minimal $N^\ast\geq 1$ such that $(N^\ast+1) \delta>\Delta$. Also there is an $N(\theta)\leq N^\ast$ such that $N\delta\leq \theta$ and $(N+1)\delta>\theta$. Write 
	\[
	\int_{t-\theta}^t f(s)\,ds = \sum_{j=1}^N \int_{t-\theta+(j-1)\delta}^{t-\theta+j\delta} f(s)\,ds + \int_{t-(\theta-N\delta)}^t f(s)\,ds.
	\]
	Taking the triangle inequality, each term in the sum is bounded by $C\gamma(t-\theta+j\delta)$, which in turn is bounded by $C\gamma(t)$. Likewise, the last term is bounded by $C\gamma(t)$, since the interval is of length less than $\delta$. Thus
	for any $\theta\in [0,\Delta]$ and $t\geq T+2\Delta$, we have 
	\[
	\left|\int_{t-\theta}^t f(s)\,ds\right|\leq (N+1)C\gamma(t)\leq (N^\ast+1)C\gamma(t)
	\leq \left(\Delta/\delta+1\right)C\gamma(t)=:K\gamma(t).
	\]
	Therefore, by dominated convergence, we have as before that $I(t)/\gamma(t)\to 0$ as $t\to\infty$, and $f_\Delta(t)=\text{o}(\gamma(t))$ by hypothesis. Hence for $t\geq T(\epsilon)$ we have $|I(t)|\leq \epsilon \gamma(t)$ and $|f_\Delta(t)|/\Delta\leq \epsilon \gamma(t)$. Set $J(T)= \int_0^T e^{s}(|f_\Delta(s)|/\Delta+|I(s)|)\,ds$. Hence by \eqref{eq.yrepave} for $t\geq T$ we have 
	\begin{align*}
		|y(t)|&\leq |\xi|e^{-t} +\int_0^T e^{-(t-s)}\frac{1}{\Delta}|f_1(s)|\,ds + \int_T^t  e^{-(t-s)}\frac{1}{\Delta}|f_1(s)|\,ds + |I(t)| \\ &\qquad+\int_0^T e^{-(t-s)} |I(s)|\,ds + \int_T^t e^{-(t-s)} |I(s)|\,ds\\
		& \leq e^{-t} (|\xi|+J(T)) 
		+ 2\epsilon \int_T^t  e^{-(t-s)}\gamma(s)\,ds + \epsilon \gamma(t)\\
		&\leq  e^{-t} (|\xi|+J(T))
		+ 2\epsilon \int_0^{t-T}  e^{-u}\,du \cdot \gamma(t) + \epsilon \gamma(t)\\
		&\leq 3\epsilon \gamma(t)+ e^{-t}(|\xi|+J(T)). 
	\end{align*}
	Hence $\limsup_{t\to\infty} |y(t)|/\gamma(t)\leq 3\epsilon$, so as $\epsilon>0$ is arbitrary, we have $y(t)=o(\gamma(t))$ as $t\to\infty$, as claimed.  
	
	We have hence sketched the important details of the following result.
	\begin{theorem} \label{lemma.fthetaymonotonelittleo}
		Let $f\in C([0,\infty);\mathbb{R})$, and $y$ be the unique continuous solution to \eqref{eq.y}. 
		Let $\gamma$ be positive, non--decreasing and in $C(0,\infty)$. Then the following are equivalent:
		\begin{enumerate}
			\item[(A)] There exists $\Delta>0$ such that  $\int_{t-\theta}^t f(s)\,ds=\text{o}(\gamma(t))$ as $t\to\infty$ for all $\theta\in [0,\Delta]$;
			\item[(B)]  For every $\Delta>0$, $\int_{t-\theta}^t f(s)\,ds=\text{o}(\gamma(t))$ as $t\to\infty$ for all $\theta\in [0,\Delta]$;
			\item[(C)] $y(t)=\text{o}(\gamma(t))$, as $t\to\infty$;	
		\end{enumerate}
	\end{theorem}
	We remark, that with $\gamma(t)=1$ we have proven from scratch the result of Grippenberg, Londen, and Staffans that $y(t)\to 0$ as $t\to\infty$ if and only if  $\int_{t-\theta}^t f(s)\,ds\to 0$ as $t\to\infty$ for all $\theta>0$ (albeit by reusing their strategy more or less wholesale). 
	
	From this result, an analogue to Theorem~\ref{lemma.fthetaysubexponentialexact} can be proven, with Theorem~\ref{lemma.fthetaymonotone} also in support. The proof is essentially identical to that of Theorem~\ref{lemma.fthetaysubexponentialexact} so is omitted.
	
	\begin{theorem} \label{lemma.fthetaymonotoneexact}
		Let $f\in C([0,\infty);\mathbb{R})$, and $y$ be the unique continuous solution to \eqref{eq.y}. 
		Let $\gamma$ be positive, non--decreasing and in $C(0,\infty)$. Then the following are equivalent:
		\item[(A)] There exists a $\Delta>0$ such that there is a $\theta'\in (0,\Delta]$ and a $C>0$ such that 
		\[
		\limsup_{t\to\infty} \frac{|\int_{t-\theta'}^t f(s)\,ds|}{\gamma(t)}>0,
		\]
		and $\left|\int_{(t-\theta)^+}^t f(s)\,ds \right|\leq K \gamma(t)$ for all $\theta\in [0,\Delta]$, $t\geq 0$;
		\item[(B)] For every $\Delta>0$ there is $\theta'=\theta'(\Delta)\in [0,\Delta]$ and a $K=K(\Delta)>0$ such that 
		\[
		\limsup_{t\to\infty} \frac{|\int_{t-\theta'}^t f(s)\,ds|}{\gamma(t)}>0,
		\]
		and $\left|\int_{(t-\theta)^+}^t f(s)\,ds \right|\leq K \gamma(t)$ for all $\theta\in [0,\Delta]$, $t\geq 0$;
		\item[(C)] 
		\[
		\limsup_{t\to\infty} \frac{|y(t)|}{\gamma(t)} \in (0,\infty).
		\]
	\end{theorem}
	
	An important corollary to Theorem~\ref{lemma.fthetaymonotoneexact} holds in the case when $\gamma(t)=1$, since it classifies exactly the situations in which $y$ is bounded, but does not tend to zero.
	\begin{theorem} \label{lemma.fthetaymonotoneexactbddnotzero}
		Let $f\in C([0,\infty);\mathbb{R})$, and $y$ be the unique continuous solution to \eqref{eq.y}. Then the following are equivalent:
		\item[(A)] There exists $\theta'>0$ such that 
		\[
		\limsup_{t\to\infty} \left|\int_{t-\theta'}^t f(s)\,ds\right|>0
		\]
		while $\left|\int_{(t-\theta)^+}^t f(s)\,ds \right|\leq K$ for all $\theta\in [0,1]$, $t\geq 0$;
		\item[(B)] 
		\[
		\limsup_{t\to\infty} |y(t)| \in (0,\infty).
		\]
	\end{theorem}
	
	\subsection{Growth order of $f_\theta$ is $\theta$--independent}	
	The first condition in part (A) of Theorems~\ref{lemma.fthetaymonotoneexact}, \ref{lemma.fthetaymonotoneexactbddnotzero} and \ref{lemma.fthetaysubexponentialexact} seems on the one hand curious, and on the other, potentially hard to check. Addressing the curious character of the condition first, it suggests that the solution attains a growth rate $\gamma$ by virtue of the perturbation size (measured through $f_\theta$) being of size $\gamma$ at a \textit{single time lag}, $\theta'$, while intuition leads us to believe that one time lag ought to be as good as another. In fact, this condition is only superficially strange, since it turns out that the condition is equivalent to a much stronger one, namely that the largest order of $f_\theta$ is $\gamma$ for almost all $\theta$; to be specific, it is equivalent to 
	\[
	\text{Leb}\left(\theta\in [0,\Delta]:
	\limsup_{t\to\infty} \frac{\left|\int_{t-\theta}^t f(s)\,ds\right|}{\gamma(t)}=0\right)=0, \quad \Delta>0.
	\]
	The new condition also addresses the potential pitfall of condition (A), namely that it would be hard to check, since, taken at face value, one would need to find the (possibly unique) value of $\theta=\theta'$ at which the growth rate $\gamma$ of $f_\theta$ was attained. Instead, the new equivalent condition asserts that we are essentially free to check the limsup at any time lag, with a good prospect that the growth rate will be identified at that lag.    
	
	\begin{theorem} \label{lemma.fthetaymonotoneexactbettercondition}
		Let $f\in C([0,\infty);\mathbb{R})$, and $y$ be the unique continuous solution to \eqref{eq.y}. 
		Let $\gamma$ be positive and in $C(0,\infty)$, and either non--decreasing or subexponential. Then the following are equivalent:
		\item[(A)] There exists a $\Delta>0$ such that there is a $\theta'\in (0,\Delta]$ and a $C>0$ such that 
		\[
		\limsup_{t\to\infty} \frac{|\int_{t-\theta'}^t f(s)\,ds|}{\gamma(t)}>0,
		\]
		and $\left|\int_{(t-\theta)^+}^t f(s)\,ds \right|\leq K \gamma(t)$ for all $\theta\in [0,\Delta]$, $t\geq 0$;
		\item[(B)] For every $\Delta>0$ there is $\theta'=\theta'(\Delta)\in [0,\Delta]$ and a $K=K(\Delta)>0$ such that 
		\[
		\limsup_{t\to\infty} \frac{|\int_{t-\theta'}^t f(s)\,ds|}{\gamma(t)}>0,
		\]
		and $\left|\int_{(t-\theta)^+}^t f(s)\,ds \right|\leq K \gamma(t)$ for all $\theta\in [0,\Delta]$, $t\geq 0$;
		\item[(C)] For every $\Delta>0$ there is a $K=K(\Delta)>0$ such that 
		\[
		\left|\int_{(t-\theta)^+}^t f(s)\,ds \right|\leq K \gamma(t), \quad \theta\in [0,\Delta], t\geq 0 
		\]
		and  
		\[
		\text{Leb}\left(\theta\in [0,\Delta]:
		\limsup_{t\to\infty} \frac{\left|\int_{t-\theta}^t f(s)\,ds\right|}{\gamma(t)}=0\right)=0; 
		\]
		\item[(D)] 
		\[
		\limsup_{t\to\infty} \frac{|y(t)|}{\gamma(t)} \in (0,\infty).
		\]
	\end{theorem}
	Another merit of condition (C) is that it essentially only necessary to establish that $f_\theta$ is of order $\gamma$ for $\theta$ on an arbitrarily short interval of the form $[0,\Delta]$. 
	\begin{proof}
		We have already shown that (A), (B) and (D) are equivalent. (C) implies (B) (by simply selecting a $\theta'$ not in the zero measure set), so (C) implies (D). Therefore, if we can show that (D) implies (C), we will be done. The rest of the proof is devoted to this task. 
		
		From (D), the first statement in (C) holds. Using this and the fact that $f$ and $\gamma$ are continuous, we see that the limit 
		\[
		L(\theta):=\limsup_{t\to\infty} \frac{\left|\int_{t-\theta}^t f(s)\,ds\right|}{\gamma(t)}
		\]  
		is finite for each $\theta\in [0,\Delta]$ and moreover $L:[0,\Delta]\to [0,\infty)$ is measurable (in fact, it is easy to show that $L$ can be extended to all of $[0,\infty)$, and that $L:[0,\infty)\to[0,\infty)$ is measurable). This implies that the set 
		\[
		Z=\{\theta\in [0,\Delta]:L(\theta)=0\}
		\] 
		is also measurable. We need to prove that $\text{Leb}(Z)=0$. To start, we show that $L$ is subadditive on $[0,\infty)$; let $\theta_1, \theta_2\geq 0$. Then for $t\geq \theta_1+\theta_2$ we have 
		\[
		\frac{\int_{t-(\theta_1+\theta_2)}^t f(s)\,ds}{\gamma(t)}
		=\frac{\int_{t-(\theta_1+\theta_2)}^{t-\theta_1} f(s)\,ds}{\gamma(t-\theta_1)}\frac{\gamma(t-\theta_1)}{\gamma(t)}
		+
		\frac{\int_{t-\theta_1}^t f(s)\,ds}{\gamma(t)}.
		\]
		Since $\gamma$ is either subexponential or increasing, the quotient involving $\gamma$ has limsup less than or equal to 1 as $t\to\infty$. Taking the absolute value across the above identity, using the triangle inequality, letting $t\to\infty$, and using the definition of $L$, we arrive at $L(\theta_1+\theta_2)\leq L(\theta_2)+L(\theta_1)$.  
		
		Next we notice that (D) implies the second statement in (B), so there is a $\theta'\in (0,\Delta]$ such that $L(\theta')>0$. Now, let $\theta\in (0,\theta'/2)$; then, using the subadditivity of $L$, we get 
		\[
		0<L(\theta')\leq L(\theta'-\theta)+L(\theta).
		\]
		Thus, for each $\theta\in (0,\theta'/2)$ at least one of $L(\theta'-\theta)$ and $L(\theta)$ is positive, which is to say, at least one of $\theta$ and $\theta-\theta'$ is in the complement of $Z$, $Z^c$. Thus, as $Z$ is measurable, so must be $Z^c$, and $\text{Leb}(Z^c)\geq \theta'/2>0$. 
		
		With these preliminaries in place, we show that $\text{Leb}(Z)=0$. Suppose, by way of contradiction, that $\text{Leb}(Z)>0$. Then, by a result of Steinhaus, the set 
		$Z-Z:=\{\theta_2-\theta_1:\theta_1,\theta_2\in Z\}$ must contain an interval of the form $(-\delta,\delta)$. Next, take $\theta_2, \theta_1\in Z$ (note that $\theta_1,\theta_2\in [0,\Delta]$) and without loss of generality, let $\theta_2>\theta_1$, and note that $\theta_2-\theta_1\in [0,\Delta]$. Now, for $t\geq \Delta$, write 
		\[
		\frac{\int_{t-(\theta_2-\theta_1)}^t f(s)\,ds}{\gamma(t)}
		=\frac{\int_{t-\theta_2}^t f(s)\,ds}{\gamma(t)}
		-
		\frac{\int_{t-(\theta_2-\theta_1)-\theta_1}^{t-(\theta_2-\theta_1)} f(s)\,ds}{\gamma(t-(\theta_2-\theta_1))}
		\frac{\gamma(t-(\theta_2-\theta_1))}{\gamma(t)}.
		\]
		Since $\theta_2$ and $\theta_1\in Z$, the quotients involving the integrals on the righthand side have zero limit as $t\to\infty$. On the other hand, the quotient involving $\gamma$ has limsup less than or equal to one, if $\gamma$ is non--decreasing or subexponential. Therefore, taking absolute values on both sides of the identity, using the triangle inequality, taking the limsup as $t\to\infty$, and using the above observations, we get 
		\[
		\lim_{t\to\infty} 	\frac{\int_{t-(\theta_2-\theta_1)}^t f(s)\,ds}{\gamma(t)}=0.
		\]
		Since $\theta_1$ and $\theta_2$ were chosen arbitrarily in $Z$, we have that $L(\theta)=0$ for any $\theta\in Z-Z$, and since $Z-Z$ contains an interval $(-\delta,\delta)$ we have that 
		\[
		\lim_{t\to\infty} \frac{\int_{t-\theta}^t f(s)\,ds}{\gamma(t)}=0, \quad \theta\in [0,\delta),
		\] 
		or that $L(\theta)=0$ for all $\theta\in [0,\delta)$, for some $\delta>0$. 
		
		If $\delta\geq \Delta$, we have that $L(\theta)=0$ for all $\theta\in [0,\Delta]$, but this is impossible, since it would force $\text{Leb}(Z^c)=0$, and we have already shown that $\text{Leb}(Z^c)\geq \theta'/2>0$. Therefore, we must have $\delta<\Delta$. On account of this, there exists an $N\in\mathbb{N}$ such that $N\delta/2\leq \Delta$ and $(N+1)\delta/2>\Delta$. Let $\theta\in [0,\Delta]$ be arbitrary. Then there is a maximal $N(\theta)\leq N$ such that $N(\theta)\delta/2\leq \theta$, and so for $t\geq \theta$, we arrive at the identity
		\begin{multline*}
			\frac{\int_{t-\theta}^t f(s)\,ds}{\gamma(t)}
			=\sum_{j=1}^{N(\theta)} \frac{\int_{t-j\delta/2}^{t-(j-1)\delta/2} f(s)\,ds}{\gamma(t-(j-1)\delta/2)} \cdot\frac{\gamma(t-(j-1)\delta/2)}{\gamma(t)}
			\\+\frac{\int_{t-\theta}^{t-N(\theta)\delta/2} f(s)\,ds}{\gamma(t-N(\theta)\delta/2)}\cdot 
			\frac{\gamma(t-N(\theta)\delta/2)}{\gamma(t)}.
		\end{multline*}
		Because each of the integrals on the righthand side have intervals of length less than or equal to $\delta/2<\delta$, each of the finitely many quotients involving these integrals has zero limit as $t\to\infty$. On the other hand, all the quotients involving $\gamma$'s have limsup less than or equal to 1, using the fact that $\gamma$ is either non--decreasing or subexponential. Therefore, for any choice of $\theta\in [0,\Delta]$, we have shown that 
		\[
		\lim_{t\to\infty} \frac{\int_{t-\theta}^t f(s)\,ds}{\gamma(t)}=0, \quad \theta\in [0,\Delta].
		\]
		Therefore $Z=\{\theta\in [0,\Delta]:L(\theta)=0\}=[0,\Delta]$, so $\text{Leb}(Z^c)=0$. But we have already shown that $\text{Leb}(Z^c)\geq \theta'/2>0$, forcing a contradiction of the original supposition that $\text{Leb}(Z)>0$. Hence we must have $\text{Leb}(Z)=0$, as claimed.  
	\end{proof}
	
	Although the result is reassuring, because it shows that windows of length $\theta$ for which the limsup is positive and finite occur for $\theta\in [0,\Delta]$ a.e., it seems a little paradoxical that windows of length $\theta$ which give rise to $f_\theta(t)=o(\gamma(t))$ cannot be ruled out entirely. We view this as an artifact of the method of proof of the above result, rather than a reflection of a deeper reality. One reason to believe this conjecture is that it is in fact true, if we restrict attention to the cases where the weight function $\gamma$ grows exponentially or superexponentially.  
	\begin{theorem} \label{thm.Zisempty}
		Suppose that $\gamma$ is a positive, continuous and non--decreasing function obeying $\lim_{t\to\infty} \gamma'(t)/\gamma(t)\in (0,\infty]$. Then (A), (B), (D) in Theorem~\ref{lemma.fthetaymonotoneexactbettercondition}	are equivalent statements, which are also equivalent to 
		\begin{enumerate}
			\item[($\text{C}'$)] For every $\Delta>0$ there is a $K=K(\Delta)>0$ such that 
			\[
			\left|\int_{(t-\theta)^+}^t f(s)\,ds \right|\leq K \gamma(t), \quad \theta\in [0,\Delta], t\geq 0 
			\]
			and for all $\theta\in [0,\Delta]$
			\[
			\limsup_{t\to\infty} \frac{\left|\int_{t-\theta}^t f(s)\,ds\right|}{\gamma(t)}>0.
			\]
		\end{enumerate}
	\end{theorem}
	\begin{proof}
		It is enough to show that (D) implies ($\text{C}'$). Suppose by way of contradiction that there exists $\theta>0$ such that $F(t):=\int_{t-\theta}^t f(s)\,ds = o(\gamma(t))$ as $t\to\infty$. Define $J(t)=\int_{t-\theta}^t y(s)\,ds$ for $t\geq \theta$. Then we have $J'(t)=-J(t)+F(t)$ for $t\geq \theta$. Since $F(t)=o(\gamma(t))$ as $t\to\infty$ and $\gamma$ is non--decreasing, we can apply Theorem~\ref{lemma.fthetaymonotonelittleo} to $J$ to  show that $J(t)=o(\gamma(t))$ as $t\to\infty$. Therefore $J'(t)=o(\gamma(t))$ as $t\to\infty$, or $y(t)-y(t-\theta)=o(\gamma(t))$ as $t\to\infty$. Therefore, for every $\epsilon>0$ there exists $T(\epsilon)>0$ such that
		\begin{equation} \label{eq.delyogamma}
			|y(t)-y(t-\theta)|<\epsilon \gamma(t),\quad t\geq T(\epsilon). 
		\end{equation}
		Consider the case where $\lim_{t\to\infty} \gamma'(t)/\gamma(t)=+\infty$. Then $\gamma(t-\theta)/\gamma(t)\to 0$ as $t\to\infty$. Since $\limsup_{t\to\infty} |y(t)|/\gamma(t)<+\infty$ by hypothesis, we have that
		\[
		\frac{y(t-\theta)}{\gamma(t)}
		=\frac{y(t-\theta)}{\gamma(t-\theta)}\cdot \frac{\gamma(t-\theta)}{\gamma(t)}\to 0, \quad t\to\infty.
		\]
		Therefore
		\[
		\frac{|y(t)|}{\gamma(t)}\leq \frac{|y(t)-y(t-\theta)|}{\gamma(t)}
		+\frac{|y(t-\theta)|}{\gamma(t)}\to 0, \quad t\to\infty,
		\]
		so $y(t)=o(\gamma(t))$ as $t\to\infty$, contradicting  the hypothesis $\limsup_{t\to\infty} |y(t)|/\gamma(t)>0$; therefore, the original supposition that $\int_{t-\theta}^t f(s)\,ds=o(\gamma(t))$ as $t\to\infty$, for some $\theta>0$, must be false. Hence in the case when $\gamma'(t)/\gamma(t)\to \infty$ as $t\to\infty$, the result is proven. 
		
		We now turn to the other case, wherein $\gamma'(t)/\gamma(t)\to \mu\in (0,\infty)$ as $t\to\infty$, noting that the estimate \eqref{eq.delyogamma} still prevails. Next let $t\geq T(\epsilon)+\theta$. Define $n=n(t)\in\mathbb{N}$ such that $n(t)\theta\leq t$ and $(n(t)+1)\theta>t$. Then there exists $t_0=t_0(t)\in [0,\theta)$ such that $t=t_0+n\theta$. Let $N\in\mathbb{N}$ be so large that $(N+1)\theta\geq T(\epsilon)$ but $N\theta<T(\epsilon)$. Thus, if $t\geq T+\theta$, we have  
		$(n(t)+1)\theta>t\geq T+\theta>(N+1)\theta$, so $n(t)>N$, or equivalently, $n(t)\geq N+1$. Write for such an arbitrary $t\geq T(\epsilon)+\theta$ the telescoping sum  
		\[
		y(t)=y(t_0+n\theta)=y(t_0+N\theta)+\sum_{j=N+1}^n \{y(t_0+j\theta)-y(t_0+(j-1)\theta)\}
		\] 
		Note for each $j\geq N+1$ that 
		$t_0+j\theta\geq t_0+(N+1)\theta\geq (N+1)\theta\geq T(\epsilon)$. Therefore, we have that 
		\[
		|y(t_0+j\theta)-y(t_0+(j-1)\theta)|<\epsilon \gamma(t_0+j\theta).
		\]
		Hence for $t\geq T(\epsilon)+\theta$, we have 
		\[
		|y(t)|\leq |y(t_0+N\theta)|+\sum_{j=N+1}^n
		\frac{\epsilon}{\theta} \gamma(t_0+j\theta)\theta.
		\]
		Since $\gamma$ is increasing, $\gamma(t_0+j\theta)\theta\leq \int_{t_0+j\theta}^{t_0+(j+1)\theta} \gamma(s)\,ds$. Hence for $t\geq T+\theta$, we get
		\[
		|y(t)|\leq 
		|y(t_0+N\theta)|+\frac{\epsilon}{\theta} 
		\int_{t_0+(N+1)\theta}^{t_0+(n(t)+1)\theta} \gamma(s)\,ds,
		\]
		or 
		\[
		|y(t)|\leq 	|y(t_0+N\theta)|+\frac{\epsilon}{\theta} 
		\int_{t_0+(N+1)\theta}^{t+\theta} \gamma(s)\,ds.
		\]
		Since $\gamma'(t)/\gamma(t)\to \mu$ as $t\to\infty$ implies $\gamma(t+\theta)/\gamma(t)\to e^{\mu \theta}$ as $t\to\infty$, we may apply L'H\^opital's rule to the indeterminate limit 
		\[
		\lim_{t\to\infty} \frac{\int_{t_0+(N+1)\theta}^{t+\theta} \gamma(s)\,ds}{\gamma(t)}
		\]
		which is of the form $\infty/\infty$ to get 
		\[
		\lim_{t\to\infty} \frac{\int_{t_0+(N+1)\theta}^{t+\theta} \gamma(s)\,ds}{\gamma(t)}
		=\lim_{t\to\infty} \frac{\gamma(t+\theta)}{\gamma'(t)}
		=\lim_{t\to\infty} \frac{\gamma(t+\theta)}{\gamma(t)}\cdot\frac{\gamma(t)}{\gamma'(t)}=\frac{e^{\mu \theta}}{\mu}.
		\]
		Therefore as $\gamma(t)\to\infty$ as $t\to\infty$, we get 
		\[
		\limsup_{t\to\infty} \frac{|y(t)|}{\gamma(t)}\leq \frac{\epsilon}{\theta} \cdot \frac{e^{\mu \theta}}{\mu}.
		\]
		Since $\epsilon>0$ is arbitrary, letting it tend to zero gives $y(t)=o(\gamma(t))$ as $t\to\infty$, again contradicting the hypothesis that $\limsup_{t\to\infty} |y(t)|/\gamma(t)>0$. Once again, this means the original supposition that $\int_{t-\theta}^t f(s)\,ds=o(\gamma(t))$ as $t\to\infty$, for some $\theta>0$, must be false, completing the proof in the case of exponential $\gamma$.
	\end{proof}

		\section{Finite--Dimensional Deterministic Equation}
		As promised at the start, we can apply the scalar deterministic results to more complicated systems. In this section, we consider the multidimensional deterministic ODEs, while in Part II, we consider a classification of the mean square of a finite dimensional affine stochastic differential equation which is both deterministically and stochastically forced.
		
		We turn now to the finite--dimensional deterministic ODE, given by    
		\begin{equation} \label{eq.xmult}
			x'(t)=Ax(t)+F(t), \quad t\geq 0; \quad x(0)=\zeta
		\end{equation}
		where the (unique continuous) solution is $\mathbb{R}^d$--valued, $A$ is a $d\times d$ matrix all of whose eigenvalues have negative real parts, $F\in C([0,\infty);\mathbb{R}^d)$ and $\zeta\in \mathbb{R}^d$.
		
		\subsection{Notation, assumptions and representations}
		Rather than write down a variation of constants formula for \eqref{eq.xmult} directly, introduce the scalar equations $i=1,\ldots,d$: 
		\[
		y_i'(t)=-y_i(t)+F_i(t), \quad t\geq 0; \quad y_i(0)=0,
		\] 
		where $F_i$ is the $i$--th component of $F$. Letting $y$ be the $\mathbb{R}^d$--vector with $i$--th component $y_i$, we have 
		\begin{equation} \label{eq.yvector}
			y'(t)=-I_d y(t)+F(t) = -y(t)+F(t), \quad t\geq 0; \quad y(0)=0,
		\end{equation}
		where $I_d$ denotes the $d\times d$ identity matrix. Define $z(t)=x(t)-y(t)$ for $t\geq 0$ and let $G(t)=(I_d+A)Y(t)$ for $t\geq 0$. Then 
		\[
		z'(t) = Az(t)+Ay(t) + I_d y(t) = Az(t)+G(t), \quad t\geq 0; \quad z(0)=\zeta.
		\]
		Let $\Phi$ be the fundamental matrix solution associated with \eqref{eq.xmult} in the case $F\equiv 0$. Therefore $\Phi$ is the unique continuous $d\times d$ matrix--valued function which is the solution to  
		\begin{equation} \label{eq.fms}
			\Phi'(t)=A\Phi(t) = \Phi(t)A, \quad t\geq 0, \quad \Phi(0)=I_d.
		\end{equation}
		In this framework, we have that 
		\[
		z(t)= \Phi(t)\zeta + \int_0^t \Phi(t-s)G(s)\,ds, \quad t\geq 0,
		\]
		so 
		\begin{equation} \label{eq.xrepymultid}
			x(t) = y(t)+\Phi(t)\zeta+  \int_0^t \Phi(t-s)(I_d+A)y(s)\,ds, \quad t\geq 0.
		\end{equation}
		Notice that this allows us to write the $F$--dependence in $x$ purely through $y$. On the other hand, writing 
		$x'(t)=-x(t)+F(t)+(I_d+A)x(t)$ for $t\geq 0$, 
		we have that 
		\[
		x(t)= e^{-t}\zeta+y(t)+\int_0^t e^{-(t-s)}(I_d+A)x(s)\,ds
		\]
		which rearranges to give
		\begin{equation}  \label{eq.yrepxmultid}
			y(t)=x(t)-e^{-t}\zeta-\int_0^t e^{-(t-s)}(I_d+A)x(s)\,ds, \quad t\geq 0.
		\end{equation}
		Notice that the assumption that the real parts of the eigenvalues of $A$ are negative leads to the exponential estimate
		\[
		\|\Phi(t)\|\leq Ke^{-\alpha t}, \quad t\geq 0
		\]
		for some $K>0$ and $\alpha>0$ (both $t$--independent). Here $\|\cdot\|$ stands for any matrix norm; we choose one which is submultiplicative: that is to say, for any $d\times d$ matrices with real entries, we have   
		\[
		\|A B\|\leq \|A\|\|B\|.
		\] 
		We use the notation $\|x\|$ for a norm of $x$ in $\mathbb{R}^d$, noting that norm equivalence allows us to pick out the behaviour of each component in the $y$'s, using the $1$--norm, for instance:  $\|y(t)\|=\sum_{i=1}^d |y_i(t)|$. Finally, we note that the matrix norms
		\[
		\|A\|=\sup_{\|x\|_p=1} \|Ax\|_p
		\] 
		induced from the $p$--norms (where $p\geq 1$)
		\[
		\|x\|_p = \left(\sum_{j=1}^d |x_j|^p\right)^{1/p}
		\] 
		have the submultiplicative property.
		
		\subsection{Main results}
		With this notation and representations in place, we can prove the desired results. We start with considering when solutions are $\text{O}(\gamma(t))$ as $t\to\infty$.  
		
		\begin{theorem} \label{thm.multidbigO}
			Let $F$ be continuous, $\gamma$ be continuous and non--decreasing. Suppose all the eigenvalues of $A$ have negative real parts. Let $x$ be the unique continuous solution of \eqref{eq.xmult}. Then the following are equivalent:
			\begin{enumerate}
				\item[(A)] There exists $\Delta>0$ and $K>0$ such that 
				\[
				\left\| \int_{(t-\theta)^+}^t F(s)\,ds \right\|\leq K\gamma(t), \quad \theta\in [0,\Delta], \quad t\geq 0;
				\]
				\item[($\text{A}'$)] For every $\Delta>0$, there is a $K=K(\Delta)>0$ such that 
				\[
				\left\| \int_{(t-\theta)^+}^t F(s)\,ds \right\|\leq K\gamma(t), \quad \theta\in [0,\Delta], \quad t\geq 0;
				\]
				\item[(B)] $y(t)=\text{O}(\gamma(t))$ as $t\to\infty$; 
				\item[(C)]  $x(t)=\text{O}(\gamma(t))$ as $t\to\infty$. 
			\end{enumerate}
		\end{theorem}
		\begin{proof}
			We show (C) implies (B); since $\|x(t)\|\leq K\gamma(t)$ for $t\geq 0$, we apply the triangle inequality to \eqref{eq.yrepxmultid} to get 
			\[
			\|y(t)\|\leq \|x(t)\|+e^{-t}\|\zeta\|+\int_0^t e^{-(t-s)}\|I_d+A\|\|x(s)\|\,ds, \quad t\geq 0.	
			\]
			Using the monotonicity, we get
			\[
			\|y(t)\|\leq K\gamma(t)+e^{-t}\|\zeta\|+K\|I_d+A\|\int_0^t e^{-(t-s)}\,ds \cdot \gamma(t).	
			\]
			Since $e^{-t}\leq 1\leq \gamma(t)/\gamma(0)$, all terms on the right are  $\text{O}(\gamma(t))$ as $t\to\infty$. 
			
			To show (B) implies (C), start with $\|y(t)\|\leq L\gamma(t)$. Applying the triangle inequality to \eqref{eq.xrepymultid}, and then using the estimates on $y$ and $\Phi$ yields
			\begin{align*}
				\|x(t)\| & \leq \|y(t)\|+\|\Phi(t)\|\|\zeta\|+  \int_0^t \|\Phi(t-s)\|\|(I_d+A)\|\|y(s)\|\,ds\\
				& \leq L\gamma(t)+Ke^{-\alpha t}\|\zeta\|+  \int_0^t Ke^{-\alpha(t-s)}\|I_d+A\|L\gamma(s)\,ds\\
				&\leq L\gamma(t)+K\|\zeta\| \frac{\gamma(t)}{\gamma(0)}+  K\frac{1}{\alpha}\|I_d+A\|L\gamma(t),
			\end{align*}
			so $x(t)=\text{O}(\gamma(t))$, as required. 
			
			The proof that (A) and ($\text{A}'$) are equivalent is similar to that showing statements (A) and (B) are equivalent in Theorem~\ref{lemma.fthetaysubexponential}, so we omit it. 
			To show (A) and (B) are equivalent, note by considering the 1--norm that (A) is equivalent to 
			having 
			\[
			\left|\int_{(t-\theta)^+}^t F_i(s)\,ds \right|\leq k \gamma(t), \quad t\geq 0, \quad \theta\in [0,\Delta],
			\]
			for some $k>0$, where $|\cdot|$ is nothing but the scalar absolute value.
			Therefore, fixing $i=1,\ldots, d$, and noting that $y_i$ obeys $y_i'=-y_i+F_i$, this last statement is equivalent to $|y_i(t)|\leq C_i\gamma(t)$ for all $t\geq 0$ for some constant $C_i>0$, which is in turn equivalent to $\|y(t)\|\leq C\gamma(t)$ for all $t\geq 0$, and for some $t$--independent $C>0$. Since this last statement is nothing other than (B), the final equivalence is proven. 
		\end{proof}
		
		This proof we offer as a prototype; from it, we see that all scalar results from Section 2 have direct analogues in the multi--dimensional case, and the proofs are trivial adaptations of the scalar ones. 
		
		We state some further results. 
		
		\begin{theorem} \label{thm.littleomultidimensional}
			Let $F$ be continuous, $\gamma$ be continuous and non--decreasing. Suppose all the eigenvalues of $A$ have negative real parts. Let $x$ be the unique continuous solution of \eqref{eq.xmult}. Then the following are equivalent:
			\begin{enumerate}
				\item[(A)] There exists $\Delta>0$ such that  
				\[
				\int_{t-\theta}^t F(s)\,ds = \text{o}(\gamma(t)), \quad t \to\infty,\quad \theta\in [0,\Delta];
				\]
				\item[($\text{A}'$)] For every $\Delta>0$
				\[
				\int_{t-\theta}^t F(s)\,ds = \text{o}(\gamma(t)), \quad t \to\infty,\quad \theta\in [0,\Delta];
				\]
				\item[(B)] $y(t)=\text{o}(\gamma(t))$ as $t\to\infty$; 
				\item[(C)]  $x(t)=\text{o}(\gamma(t))$ as $t\to\infty$. 
			\end{enumerate}
		\end{theorem}
		
		To get the rate of the fluctuations or growth exactly, we combine the last two results. We give the proof here so that it can be seen how one can extract information about a particular component.
		
		\begin{theorem} \label{thm.finitedimerateexact}
			Let $F$ be continuous, $\gamma$ be continuous and non--decreasing. Suppose all the eigenvalues of $A$ have negative real parts. Let $x$ be the unique continuous solution of \eqref{eq.xmult}. Then the following are equivalent:
			\begin{enumerate}
				\item[(A)]  There exist $\Delta>0$, $K>0$ such that 
				\[
				\left\| \int_{(t-\theta)^+}^t F(s)\,ds \right\|\leq K\gamma(t), \quad \theta\in [0,\Delta], \quad t\geq 0,
				\]
				and there also exist $i\in \{1,\ldots, d\}$, $\theta'>0$  such that 
				\[
				\limsup_{t\to\infty} \frac{|\int_{t-\theta'}^t F_i(s)\,ds|}{\gamma(t)} >0;
				\]
				\item[($\text{A}'$)] For every $\Delta>0$ there is $K=K(\Delta)>0$ such that 
				\[
				\left\| \int_{(t-\theta)^+}^t F(s)\,ds \right\|\leq K\gamma(t), \quad \theta\in [0,\Delta], \quad t\geq 0;
				\]	
				and there exist $i\in \{1,\ldots, d\}$, $\theta'\in (0,\Delta]$ such that  
				\[
				\limsup_{t\to\infty} \frac{|\int_{t-\theta'}^t F_i(s)\,ds|}{\gamma(t)} >0;
				\]	
				\item[($\text{A}''$)] For every $\Delta>0$ there exists $K=K(\Delta)>0$ such that 
				\[
				\left\| \int_{(t-\theta)^+}^t F(s)\,ds \right\|\leq K\gamma(t), \quad \theta\in [0,\Delta], \quad t\geq 0
				\]	
				and there also exists $i\in \{1,\ldots, d\}$ such that  
				\[
				\text{Leb}\left(\theta\in [0,\Delta]: 
				\limsup_{t\to\infty} \frac{|\int_{t-\theta}^t F_i(s)\,ds|}{\gamma(t)} =0 \right)=0;
				\]			
				\item[(B)] 
				\[
				\limsup_{t\to\infty}\frac{\|y(t)\|}{\gamma(t)} \in (0,\infty);  
				\]
				\item[(C)] 	
				\[
				\limsup_{t\to\infty}\frac{\|x(t)\|}{\gamma(t)} \in (0,\infty).  
				\]  
			\end{enumerate}
		\end{theorem}
		\begin{proof} By earlier discussions, all the conditions in (A) are equivalent, so our proof simply needs to show that (A), (B) and (C) are equivalent. 
			
			We start by showing (C) implies (B). 
			If (C) holds, then $y(t)=\text{O}(\gamma(t))$ as $t\to\infty$. Suppose by way of contradiction that $y(t)=\text{o}(\gamma(t))$ as $t\to\infty$. But then $x(t)=o(\gamma(t))$ as $t\to\infty$, contradicting (C). Hence (C) implies (B). 
			
			Conversely, if (B) holds, then $x(t)=\text{O}(\gamma(t))$ as $t\to\infty$. Suppose by way of contradiction that $x(t)=\text{o}(\gamma(t))$ as $t\to\infty$; but then $y(t)=\text{o}(\gamma(t))$ as $t\to\infty$, contradicting the hypothesis (B). Thus (B) implies (C), and we have shown that (B) and (C) are equivalent. 
			
			We now show that (A) and (B) are equivalent. If (A) holds, then $y(t)=\text{O}(\gamma(t))$ as $t\to\infty$. Suppose by way of contradiction that $y(t)=o(\gamma(t))$ as $t\to\infty$. Then each $y_i(t)=o(\gamma(t))$ as $t\to\infty$, and therefore we have for each $i=1,\ldots,d$ and each $\theta\in [0,1]$ that $\int_{t-\theta'}^t F_i(s)\,ds=o(\gamma(t))$ as $t\to\infty$. But this contradicts the first part of (A), which is true by hypothesis. Thus, $y(t)$ cannot be $\text{o}(\gamma(t))$ as $t\to\infty$ and hence (B) holds. Thus (A) implies (B). 
			
			If (B) holds, the second part of (A) holds. Suppose now by way of contradiction that the first part of (A) is false, so we have that 
			$\int_{t-\theta}^t F_i(s)\,ds = o(\gamma(t))$ as $t\to\infty$ for all $\theta\in [0,1]$ and each $i\in \{1,\ldots,d\}$. Then 
			\[
			\left\| \int_{t-\theta}^t F(s)\,ds\right\|_1 = \sum_{i=1}^d 
			\left|\left[\int_{t-\theta}^t F(s)\,ds\right]\right|
			=\sum_{i=1}^d 
			\left|\int_{t-\theta}^t F_i(s)\,ds\right|
			=\text{o}(\gamma(t)), \quad t\to\infty.
			\] 
			But this implies that $y(t)=\text{o}(\gamma(t))$ as $t\to\infty$, which contradicts (B). Hence (B) implies (A), and we have shown (A) and (B) are equivalent. 
		\end{proof}
		
		Likewise, we have a result about exact subexponential bounds (which may even deal with decay to zero).
		
		\begin{theorem} \label{lemma.fthetaxsubexponentialexact}
			Let $F$ be continuous. Suppose all the eigenvalues of $A$ have negative real parts. Let $x$ be the unique continuous solution of \eqref{eq.xmult}. 
			Let $\gamma$ be subexponential. Then the following are equivalent:
			\begin{enumerate}
				\item[(A)]  There exist $\Delta>0$, $K>0$ such that 
				\[
				\left\| \int_{(t-\theta)^+}^t F(s)\,ds \right\|\leq K\gamma(t), \quad \theta\in [0,\Delta], \quad t\geq 0,
				\]
				and there also exist $i\in \{1,\ldots, d\}$, $\theta'>0$  such that 
				\[
				\limsup_{t\to\infty} \frac{|\int_{t-\theta'}^t F_i(s)\,ds|}{\gamma(t)} >0;
				\]
				\item[($\text{A}'$)] For every $\Delta>0$ there is $K=K(\Delta)>0$ such that 
				\[
				\left\| \int_{(t-\theta)^+}^t F(s)\,ds \right\|\leq K\gamma(t), \quad \theta\in [0,\Delta], \quad t\geq 0;
				\]	
				and there exist $i\in \{1,\ldots, d\}$, $\theta'\in (0,\Delta]$ such that  
				\[
				\limsup_{t\to\infty} \frac{|\int_{t-\theta'}^t F_i(s)\,ds|}{\gamma(t)} >0;
				\]	
				\item[($\text{A}''$)] For every $\Delta>0$ there exists $K=K(\Delta)>0$ such that 
				\[
				\left\| \int_{(t-\theta)^+}^t F(s)\,ds \right\|\leq K\gamma(t), \quad \theta\in [0,\Delta], \quad t\geq 0
				\]	
				and there also exists $i\in \{1,\ldots, d\}$ such that  
				\[
				\text{Leb}\left(\theta\in [0,\Delta]: 
				\limsup_{t\to\infty} \frac{|\int_{t-\theta}^t F_i(s)\,ds|}{\gamma(t)} =0 \right)=0;
				\]			
				\item[(B)] 
				\[
				\limsup_{t\to\infty}\frac{\|y(t)\|}{\gamma(t)} \in (0,\infty);  
				\]
				\item[(C)] 	
				\[
				\limsup_{t\to\infty}\frac{\|x(t)\|}{\gamma(t)} \in (0,\infty).  
				\]  
			\end{enumerate}
		\end{theorem}
		We remark that in all the results in this section, conditions involving norms on $F$ can be replaced by conditions on the component functions $F_i$. Thus, in checking conditions for the norm of the vector--valued solution $x$, it is necessary only to check conditions on each component of $F$. We leave the formulation of such results to the interested reader.  
		
		\subsection{Derivative bounds}
		The above results have concentrated on the bounds of the solution, and how these are connected to the bounds on the average of $F$, through $F_\theta$. It is therefore resonable to ask what role bounds on $f$ itself play. Our next result shows that the bounds on $f$ are connected intimately with bounds on the derivative. We prove the result in the case of increasing $\gamma$. 
		
		We start with an obvious, but important observation. Let $\Gamma$ be positive, non--decreasing and continuous on $[0,\infty)$. Suppose that $\limsup_{t\to\infty} \|F(t)\|/\Gamma(t)<+\infty$ and $f$ is continuous. Then $\|f(t)\|\leq K\Gamma(t)$ for all $t\geq 0$. Thus for $t\geq \theta$ and $\theta\in [0,1]$ we have 
		\begin{equation} \label{eq.Fthetagamma}
			\left\|\int_{(t-\theta)^+}^t F(s)\,ds \right\|
			= \left\|\int_{t-\theta}^t F(s)\,ds \right\|
			\leq \int_{t-\theta}^t \|F(s)\|\,ds \leq \theta K\Gamma(t)\leq K\Gamma(t),
		\end{equation}
		and for $0\leq t\leq \theta$ and $\theta\in [0,1]$
		\[
		\left\|\int_{(t-\theta)^+}^t F(s)\,ds \right\|
		= \left\|\int_0^t F(s)\,ds \right\|
		\leq \int_0^{t} \|F(s)\|\,ds \leq Kt\Gamma(t)\leq K\Gamma(t).
		\]
		\begin{theorem} \label{thm.derivativebounds0}
			Let $F$ be continuous. Suppose all the eigenvalues of $A$ have negative real parts. Let $x$ be the unique continuous solution of \eqref{eq.xmult}. 
			Let $\gamma$ be positive, continuous and non--decreasing. 
			\begin{itemize}
				\item[(i)] The following are equivalent:
				\begin{enumerate}
					\item[(A)] $F(t)=O(\gamma(t))$ as $t\to\infty$;
					\item[(B)]  $x(t)=O(\gamma(t))$, $x'(t)=O(\gamma(t))$ as $t\to\infty$.
				\end{enumerate}
				\item[(ii)] The following are equivalent:
				\begin{enumerate}
					\item[(A)] $F(t)=o(\gamma(t))$ as $t\to\infty$;
					\item[(B)]  $x(t)=o(\gamma(t))$, $x'(t)=o(\gamma(t))$ as $t\to\infty$.
				\end{enumerate}
			\end{itemize}
		\end{theorem}
		\begin{proof}
			For part (i), suppose (A) holds. Then $F$ obeys $F(t)=\text{O}(\gamma(t))$ as $t\to\infty$, and, as observed above, we have that \eqref{eq.Fthetagamma} holds. Hence $x(t)=\text{O}(\gamma(t))$ as $t\to\infty$, which is the first part of (B). But $x'(t)=Ax(t)+F(t)$, so $x'(t)=\text{O}(\gamma(t))$ as $t\to\infty$, proving (B). 
			
			On the other hand, if (B) holds, write $F(t)=x'(t)-Ax(t)$, so all of the righthand side is $\text{O}(\gamma(t))$ as $t\to\infty$, proving (A). 
			
			For part (ii), to show (A) implies (B), we note that $F(t)=o(\gamma(t))$ as $t\to\infty$ implies $\int_{t-\theta}^t F(s)\,ds =o(\gamma(t))$ as $t\to\infty$. Hence $x(t)=o(\gamma(t))$ as $t\to\infty$. But then $x'=Ax+F$, so $x'(t)=o(\gamma(t))$ as $t\to\infty$. On the other hand, $F=x'-Ax$ so (B) clearly implies (A).
		\end{proof}
		
		In the case when the solutions of the differential equations are growing or fluctuating unboundedly, we can now connect very clearly the size of $x'$ with $F$ and the size of $x$ with that of $F_\theta$.  
		\begin{theorem} \label{thm.derivativebounds}
			Let $F$ be continuous. Suppose all the eigenvalues of $A$ have negative real parts. Let $x$ be the unique continuous solution of \eqref{eq.xmult}. 	
			\begin{itemize}
				\item[(i)] Suppose $\gamma$ and $\Gamma$  are positive, continuous and non--decreasing functions with $\gamma(t)=o(\Gamma(t))$ as $t\to\infty$. Then the  following are equivalent: 
				\begin{enumerate}
					\item[(A)] There exists $K>0$ and $\theta'>0$ such that 
					\begin{gather*}
						\left|\int_{(t-\theta)^+}^t F(s)\,ds \right|\leq K\gamma(t), \quad 
						\limsup_{t\to\infty} \left|\int_{t-\theta'}^t F(s)\,ds\right|/\gamma(t)>0, \\
						\limsup_{t\to\infty} |F(t)|/\Gamma(t)\in (0,\infty);
					\end{gather*}
					\item[(B)]
					\[
					\limsup_{t\to\infty} \frac{\|x(t)\|}{\gamma(t)}\in (0,\infty), \quad 
					\limsup_{t\to\infty} \frac{\|x'(t)\|}{\Gamma(t)}\in (0,\infty).
					\]
				\end{enumerate}
				Moreover, both imply $x(t)=o(\Gamma(t))$ as $t\to\infty$. 
				\item[(ii)] Suppose $\gamma$ is a positive, continuous and non--decreasing function. Then the  following are equivalent: 
				\begin{enumerate}
					\item[(A)] There exists $\theta'>0$ and $i\in \{1,\ldots,d\}$ such that 
					\begin{gather*}
						F(t)=O(\gamma(t)), \quad t\to\infty, \quad 
						\limsup_{t\to\infty} \left|\int_{t-\theta'}^t F_i(s)\,ds\right|/\gamma(t)>0;
					\end{gather*}
					\item[(B)]
					\[
					\limsup_{t\to\infty} \frac{\|x(t)\|}{\gamma(t)}\in (0,\infty), \quad x'(t)=O(\gamma(t)), \quad t\to\infty.
					\]
					\item[(iii)] If $\limsup_{t\to\infty} \|x(t)\|/\gamma(t)\in (0,\infty)$ and $x'(t)=o(\gamma(t))$ as $t\to\infty$, then 
					\[
					\limsup_{t\to\infty} \frac{|F(t)|}{\gamma(t)}\in (0,\infty).
					\]
				\end{enumerate}
			\end{itemize}
		\end{theorem} 
		\begin{proof}
			For part (i), the first part of the hypotheses in (A) give the first part of (B). The second part of (A) gives $x'(t)=O(\Gamma(t))$ as $t\to\infty$ by the last result. Moreover, since $\gamma(t)=o(\Gamma(t))$  as $t\to\infty$ and $x(t)=O(\gamma(t))$ as $t\to\infty$, we have $x(t)=o(\Gamma(t))$ as $t\to\infty$. Suppose by way of contradiction that $x'(t)=o(\Gamma(t))$ as $t\to\infty$. Then $F(t)=x'(t)-Ax(t)=o(\Gamma(t))$ as $t\to\infty$, forcing a contradiction. Hence the second part of (B) holds. 
			
			Conversely, the first part of (B) implies the first part of (A). Then, as $x(t)=O(\gamma(t))$ and $\gamma(t)=o(\Gamma(t))$ as $t\to\infty$, we have $x(t)=o(\Gamma(t))$ as $t\to\infty$. Suppose by way of contradiction that $F(t)=o(\Gamma(t))$ as $t\to\infty$. Then $x'(t)=Ax(t)+F(t)=o(\Gamma(t))$ as $t\to\infty$, contradicting the second part of (B). This completes the proof of part (i).
			
			For part (ii), note that the first part of (A) gives the first part of (A) in part (i), so therefore the first part of (B) holds, which implies $x(t)=O(\gamma(t))$. Since $F(t)=O(\gamma(t))$ by hypothesis, we have $x'(t)=Ax(t)+F(t)=O(\gamma(t))$ as $t\to\infty$. Thus (A) implies (B). 
			
			Conversely, the first part of (B) implies the limsup in part (A). By (B) and the fact that  $F(t)=x'(t)-Ax(t)$, we have that $F(t)=O(\gamma(t))$ as $t\to\infty$, proving (A).
			
			For part (iii), we have by hypothesis that $F(t)=x'(t)-Ax(t)$ is $O(\gamma(t))$ as $t\to\infty$. On the other hand, since $\|x\|/\gamma$ has a positive limsup, we have that 
			\[
			\limsup_{t\to\infty} \frac{\|\int_{t-\theta'}^t F(s)\,ds\|}{\gamma(t)}>0
			\]
			for some $\theta'>0$. Now, suppose by way of contradiction that $F(t)=o(\gamma(t))$ as $t\to\infty$. Then for every $\theta>0$ we have $\int_{t-\theta}^t F(s)\,ds = o(\gamma(t))$ as $t\to\infty$, contradicting the last displayed inequality. Hence, $F$ cannot be $o(\gamma(t))$ as $t\to\infty$, and the proof is complete.
		\end{proof}
		
		Part (i) of the theorem might leave the reader with the impression that the behaviour of the derivative is ``worse'' than the solution, or perhaps, in view of part (ii), no better. However, it can be that $\|x(t)\|/\gamma(t)$ has a positive limsup, while $x'(t)=\text{o}(\gamma(t))$ as $t\to\infty$. Consider the equation $x'=Ax+F$ where $A=-\text{diag}(\alpha_1,\ldots,\alpha_d)$, with each $\alpha_i>0$ and $F_i(t)\sim c_i\gamma_i(t)$ as $t\to\infty$, where each $\gamma_i$ is such that $\gamma_i'(t)=o(\gamma_i(t))$ and $\gamma_i(t)\to\infty$ as $t\to\infty$. Each $\gamma_i$ is subexponential and obeys $e^{\alpha_i t}\gamma_i(t)\to\infty$ as $t\to\infty$. Also, $x_i(t)$ obeys
		\[
		\frac{x_i(t)}{\gamma_i(t)}=\frac{x_i(0)}{e^{\alpha_i t} \gamma_i(t)}+ 
		\frac{\int_0^t e^{\alpha_i s}F_i(s)\,ds}{e^{\alpha_i t} \gamma_i(t)}.
		\]
		The first term on the righthand side has zero limit. Applying L'H\^opital's rule to the indeterminate limit in the second term, we get
		\[
		\lim_{t\to\infty} \frac{x_i(t)}{\gamma_i(t)}
		=\lim_{t\to\infty} \frac{e^{\alpha_i t}F_i(t)}{e^{\alpha_i t} \gamma_i'(t)+\alpha_i e^{\alpha_i t}\gamma_i(t)}
		=\frac{c_i}{\alpha_i}.
		\] 
		This implies, as $t\to\infty$, that 
		\[
		\frac{x_i'(t)}{\gamma_i(t)}
		=-\alpha_i\frac{x_i(t)}{\gamma_i(t)}+\frac{F_i(t)}{\gamma_i(t)}\to -\alpha_i\cdot\frac{c_i}{\alpha_i}+c_i=0.
		\]
		Now, for simplicity, suppose that $\gamma_i(t)/\gamma_1(t)\to 0$ as $t\to\infty$ for all $i\geq 2$, and that $c_1\neq 0$. Then, with $\gamma=\gamma_1$, we have 
		\[
		\lim_{t\to\infty} \frac{\|x_1(t)\|_1}{\gamma(t)} = \frac{|c_1|}{\alpha_1},
		\] 
		while $\|x'(t)\|_1 = o(\gamma(t))$ as $t\to\infty$.
		
		A similar example where $\Gamma$ is superexponential (in the sense that $\Gamma'(t)/\Gamma(t)\to\infty$ as $t\to\infty$, which implies $\lim_{t\to\infty} \log \Gamma(t)/t=+\infty$) gives an example in which part (i) holds. Here we assume there is $c_1\neq 0$ such that $F_1(t)\sim c_1\Gamma(t)$ and $F_i(t)=o(\Gamma(t))$ as $t\to\infty$ for $i\geq 2$. 
		Let $\gamma(t)=\int_0^t \Gamma(s)\,ds$. Then $\gamma''(t)/\gamma'(t)\to\infty$ as $t\to\infty$, so by asymptotic integration, $\gamma'(t)/\gamma(t)=\Gamma(t)/\gamma(t)\to \infty$ as $t\to\infty$. Hence, by L'H\^opital's rule, we have  
		\[
		\lim_{t\to\infty} 
		\frac{x_i(t)}{\gamma(t)}=\lim_{t\to\infty}\frac{x_i(0)}{e^{\alpha_i t} \gamma(t)}+ 
		\frac{\int_0^t e^{\alpha_i s}F_i(s)\,ds}{e^{\alpha_i t} \gamma(t)}.
		=0+\lim_{t\to\infty} 
		\frac{F_i(t)/\Gamma(t)}{1+\alpha_i \Gamma(t)/\gamma(t)}.
		\]
		Thus $x_1(t)\sim c_1\gamma(t)$ as $t\to\infty$, and $x_i(t)=o(\gamma(t))$ for $i\geq 2$. On the other hand 
		\[
		\frac{x_i'(t)}{\Gamma(t)}=-\alpha_i\frac{x_i(t)}{\gamma(t)}\frac{\gamma(t)}{\Gamma(t)}
		+\frac{F_i(t)}{\Gamma(t)}, 
		\]
		so $x_1'(t)\sim c_1 \Gamma(t)$ and $x_i'(t)=o(\Gamma(t))$ as $t\to\infty$ for $i\geq 2$. Hence we have $\gamma(t)=o(\Gamma(t))$, $\|x(t)\|_1\sim |c_1|\gamma(t)$ and $\|x'(t)\|_1\sim |c_1|\gamma(t)$ as $t\to\infty$, which is an example of the situation prevailing in part (i).
		
		\subsection{Perron--type converses} 
		We close with an example of a Perron--type theorem, which we did not consider in the scalar case. Let $\gamma$ be subexponential, and suppose for the sake of argument, that for every continuous function $f$ for which $f/\gamma$ is bounded, the solution $x$ of the differential equation \eqref{eq.xmult} is such that $x/\gamma$ is bounded. This is the situation that prevails in Theorem~\ref{lemma.fthetaxsubexponentialexact} \textit{under the assumption that all the eigenvalues of $A$ have negative real parts}. But we now show that if $x$ does enjoy such a bound, it must be that this assumption holds. 
		
		\begin{theorem} \label{thm.perronsubexponential}
			Let $\gamma$ be subexponential. Suppose for every $F$ in $\text{BC}_\psi$ that the corresponding unique continuous solution $x$ of \eqref{eq.xmult} is also in $\text{BC}_\psi$. Then all the eigenvalues of $A$ have negative real parts. 	
		\end{theorem}
		\begin{proof}
			To see this, let $x(0)=0$ and recall we may take $\gamma$ to be $C^1$ with enhanced properties: if we have a continuous $\gamma$ then $\Gamma(t)=\int_{t}^{t+1}\gamma(s)\,ds$ is subexponential, in $C^1$, obeys $\Gamma(t)\sim \gamma(t)$ as $t\to\infty$ and $\Gamma'(t)/\Gamma(t)\to 0$ as $t\to\infty$. Let $\gamma$ possess all these properties. This approach is legitimate since if $\gamma_1$ and $\gamma_2$ are continuous weight functions with $\gamma_1(t)\sim \gamma_2(t)$ as $t\to\infty$, then $BC_{\gamma_1}=BC_{\gamma_2}$. 
			
			Write $\tilde{F}(t)=F(t)/\gamma(t)$ and $\tilde{x}(t)=x(t)/\gamma(t)$. Then we have 
			\[
			\tilde{x}'(t)=A(t)\tilde{x}(t)+\tilde{F}(t), \quad t\geq 0
			\] 
			with $A(t):=A-\gamma'(t)/\gamma(t)\cdot I_d$. Thus $A(\cdot)$ is continuous and bounded (indeed $A(t)\to A$ as $t\to\infty$). Let the resolvent $R$ be given by 
			\begin{equation} \label{eq.resol}
				\frac{\partial}{\partial t}R(t,s)=A(t)R(t,s), \quad t>s; \quad R(s,s)=I_d.
			\end{equation}
			Then 
			\[
			\tilde{x}(t)=\int_0^t R(t,s)\tilde{F}(s)\,ds, \quad t\geq 0.
			\]
			Notice moreover that $R(t,s)=\Phi(t-s)\gamma(s)/\gamma(t)$, where $\Phi$ obeys \eqref{eq.fms}; this can be checked by showing this $R$ obeys the differential equation \eqref{eq.resol}, and since the differential equation has a unique solution, this identifies $R$ uniquely. Next by hypothesis it is true, for every bounded continuous $\tilde{F}$, that 
			\[
			t\mapsto \int_0^t R(t,s)\tilde{F}(s)\,ds 
			\] 
			is also a bounded continuous function. Since $R$ is a resolvent, it follows from Perron's result that there is an $M>0$ ($t$--independent) such that 
			\[ 
			\int_0^t |R(t,s)|\,ds \leq M, \quad t\geq 0.
			\]
			Therefore for any $T>0$ and $t\geq T$ we have 
			\[
			M\geq 
			\int_0^t |\Phi(t-s)|\frac{\gamma(s)}{\gamma(t)}\,ds = \int_0^t |\Phi(u)|\frac{\gamma(t-u)}{\gamma(t)}\,du \geq \int_0^T |\Phi(u)|\frac{\gamma(t-u)}{\gamma(t)}\,du.
			\]
			Now, since $\gamma$ is subexponential, if we let $t\to\infty$ on both sides, we get 
			\[
			\int_0^T |\Phi(u)|\,du \leq M, \quad T>0.
			\]
			Since $T$ is arbitrary, we have $\Phi\in L^1(\mathbb{R}^+;\mathbb{R}^{d\times d})$. But this can only be the case if all the eigenvalues of $A$ have negative real parts.
		\end{proof}
		
		\subsection{Perturbations of unstable equations}
		The theme of the paper up to now has revolved around ``stable'' equations which are then perturbed, and to give characterisations of their $L_\infty$ asymptotic behaviour, where the impact of the forcing term  on solutions appears to be well described by bounds on $F_\theta(t)=\int_{(t-\theta)^+}^t F(s)\,ds$. However, as we see in this short subsection, the theory also appears to extend quite naturally to equations which are ``unstable'', and are then externally forced. 
		
		Rather than trying to develop a comprehensive suite of theorems, we will prove a theorem which characterises solution sizes for ``large'' perturbations, and then give what appear relatively sharp conditions for growth rates when the perturbations are subdominant to the solution of the ordinary differential equation. 
		
		Once again, we study the differential equation \eqref{eq.xmult}, but now we assume that at least one of the eigenvalues of $A$ has a non--negative real part. Let $\sigma(A)$ be the spectrum of $A$, and let 
		\[
		\lambda(A):=\max\{\text{Re}(\lambda):\lambda\in \sigma(A)\}\geq 0.
		\]     
		Also, for each $\lambda\in \sigma(A)$ such that $\text{Re}(\lambda)=\lambda(A)$, let $\mu_A(\lambda)$ be its algebraic multiplicity. Define the maximal algebraic multiplicity among all such dominant eigenvalues by  
		\begin{equation} \label{eq.nA}
			n(\lambda(A)):=\max\{\mu_A(\lambda): \lambda\in \sigma(A), \text{Re}(\lambda)=\lambda(A)\}.
		\end{equation}
		Then $n=n(\lambda(A))\geq 1$. With this notation, and $\alpha=\lambda(A)\geq 0$, we have that there is a $C>0$ such that the fundamental matrix solution $\Phi$ of \eqref{eq.fms} obeys 
		\[
		\|\Phi(t)\|\leq C(1+t)^{n-1}e^{\alpha t}, \quad t\geq 0.
		\]
		Therefore, solutions of the unperturbed differential equation $u'(t)=Au(t)$ obey $u(t)=O(t^{n-1}e^{\alpha t})$, $t\to\infty$. Roughly speaking, it is therefore natural to consider perturbations which grow slower or faster than $e^{\alpha t}$ as $t\to\infty$. 
		
		This leads us to consider smooth weight functions $\gamma$ which obey
		\begin{equation} \label{eq.gammalebeta}
			\lim_{t\to\infty} \frac{\gamma'(t)}{\gamma(t)}=:\beta\in [0,\infty]
		\end{equation}
		The case when $\beta=+\infty$ will help us deal with superexponentially growing perturbations. Examples of such $\gamma$ include $\gamma(t)=e^{t^\eta}$ for $\eta>1$ or $\gamma(t)=\exp(e^t)$. If $\beta=0$, we have smooth subexponential functions (which can still be dominant in the case $\alpha=0$), while for $\beta>0$, $\gamma$ can be decomposed according to  $\gamma(t)=e^{\beta t}\gamma_0(t)$, where $\gamma_0$ is   
		smooth subexponential. Recall that asking for extra smoothness in weight functions is usually an irrestrictive side condition, and can ease analysis greatly. 
		
		We neglect the case when $\beta<0$, since exponentially decaying perturbations will not in general have a major effect on the relative size of solutions, given that the unperturbed equation exhibits exponential growth when $\alpha>0$.  	
		
		We give a characterisation of solutions where the unperturbed solution is subdominant (i.e. when $\alpha<\beta$). 
		
		\begin{theorem} \label{thm.unstablelargepert}
			Let $F$ be continuous. Suppose $\alpha=\lambda(A)\geq 0$ and let $n:=n(A)\geq 1$ be given by \eqref{eq.nA}. Let $x$ be the unique continuous solution of \eqref{eq.xmult}, and let $\gamma$ be a positive and $C^1$ function obeying \eqref{eq.gammalebeta}, and let $\beta>\alpha$. Then the following are equivalent:
			\begin{itemize}
				\item[(A)] There is $\Delta>0$, $K>0$ and $\theta'>0$ such that 
				\[
				\left\|\int_{(t-\theta)^+}^t  F(s)\,ds\right\|\leq K \gamma(t), \quad t\geq 0, \quad \theta\in [0,\Delta],
				\] 
				and 
				\[
				\limsup_{t\to\infty} \frac{\|\int_{t-\theta'}^t F(s)\,ds\|}{\gamma(t)}>0;
				\]
				\item[(B)] 
				\[
				\limsup_{t\to\infty} \frac{\|x(t)\|}{\gamma(t)}\in (0,\infty);
				\]
			\end{itemize} 	
		\end{theorem} 
		\begin{proof}
			Since $\beta>\alpha\geq 0$, we have that $\gamma'(t)>0$ for all $t$ sufficiently large, and that this ultimate monotonicity can be assumed in place of monotonicity in all proofs without harm. Let $y$ be the $d$--dimensional solution of \eqref{eq.yvector}. 
			
			Assume (A). Then $\limsup_{t\to\infty} \|y(t)\|/\gamma(t)\in (0,\infty)$. Thus $\|y(t)\|\leq L\gamma(t)$ for all $t\geq 0$. Moreover, by \eqref{eq.gammalebeta} we have that 
			\[
			\lim_{t\to\infty} \log \gamma(t)/t=\beta>\alpha,
			\]
			so $\Phi(t)=o(\gamma(t))$ as $t\to\infty$. As before, \eqref{eq.xrepymultid} holds viz., 
			\begin{equation*} 
				x(t) = y(t)+\Phi(t)\zeta+  \int_0^t \Phi(t-s)(I_d+A)y(s)\,ds, \quad t\geq 0.
			\end{equation*}
			Majorisation leads to 
			\[
			\|x(t)\| \leq L\gamma(t) +\|\Phi(t)\|\|\zeta\|+ LC\|I_d+A\| \int_0^t (1+(t-s))^{n-1}e^{\alpha(t-s)} \gamma(s)\,ds, \quad t\geq 0.
			\]
			Since $\beta>\alpha$, for every $\epsilon\in (0,\beta-\alpha)$, there is a $D=D(\epsilon)>0$ such that $(1+t)^{n-1}\leq D(\epsilon) e^{\epsilon t}$ for all $t\geq 0$. Hence for all $t\geq 0$ we have 
			\begin{equation} \label{eq.masterest1}
				\frac{\|x(t)\|}{\gamma(t)} \leq L +\frac{\|\Phi(t)\|}{\gamma(t)}\|\zeta\|+ LCD(\epsilon)\|I_d+A\| \frac{1}{\gamma(t)e^{-(\alpha+\epsilon)t}}\int_0^t e^{-(\alpha+\epsilon)s} \gamma(s)\,ds. 
			\end{equation}
			As already noted, the second term on the righthand side of \eqref{eq.masterest1} has zero limit. For the third term, note from \eqref{eq.gammalebeta} that the function $\gamma_\epsilon(t):=\gamma(t)e^{-(\alpha+\epsilon)t}$ is such that $\gamma_\epsilon'(t)/\gamma_\epsilon(t)\to \beta-(\alpha+\epsilon)>0$ as $t\to\infty$, so by L'H\^opital's rule
			\[
			\int_0^t \gamma_\epsilon(s)\,ds \sim \frac{1}{\beta-(\alpha+\epsilon)}\gamma_\epsilon(t), \quad t\to\infty,
			\]  
			where the integral is $\text{o}(\gamma_\epsilon(t))$ as $t\to\infty$ in the case that $\beta=+\infty$. Therefore the last term on the righthand side of \eqref{eq.masterest1} tends to a finite limit as $t\to\infty$, and we have the finiteness of the limit superior in (B). 
			
			Suppose now, by way of contradiction that $x(t)=\text{o}(\gamma(t))$ as $t\to\infty$. 
			Now, we revisit the representation \eqref{eq.yrepxmultid} to get
			\begin{equation*}  
				\frac{\|y(t)\|}{\gamma(t)}\leq \frac{\|x(t)\|}{\gamma(t)}+\frac{1}{\gamma(t)e^{t}}\|\zeta\|+\|I_d+A\| \frac{1}{ \gamma(t)e^{t}}\int_0^t e^s\|x(s)\|\,ds.
			\end{equation*}
			The first two terms on the righthand side have zero limits. For the third term, the integral is either convergent, in which case the last term tends to zero, or is divergent, in which case the last term has indeterminate limit of the form $\infty/\infty$. In this situation, apply again L'H\^opital's rule to the quotient:
			\[
			\lim_{t\to\infty} \frac{\int_0^t e^s \|x(s)\|\,ds}{e^t \gamma(t)}
			=\lim_{t\to\infty} \frac{\|x(t)\|}{\gamma'(t)+\gamma(t)}
			=\lim_{t\to\infty} \frac{\|x(t)\|/\gamma(t)}{\gamma'(t)/\gamma(t)+1}=0,
			\]
			appealing to \eqref{eq.gammalebeta} at the last step. Therefore, we have that $x(t)=\text{o}(\gamma(t))$ implies $y(t)=\text{o}(\gamma(t))$ which generates the desired contradiction. Hence (A) implies (B). 
			
			To show that (B) implies (A), note that the first part of (A) comes from the estimate $\|x(t)\|\leq K\gamma(t)$ for all $t\geq 0$ (which is a consequence of (B)) and the identity
			\begin{equation} \label{eq.Fthetax}
				\int_{t-\theta}^t F(s)\,ds = x(t)-x(t-\theta)-A\int_{t-\theta}^t x(s)\,ds, \quad t\geq \theta.
			\end{equation}
			To show the second part of (A) holds, assume by way of contradiction that it does not. Then $\int_{t-\theta}^t F(s)\,ds =o(\gamma(t))$ for all $\theta>0$. But this implies that $y(t)=\text{o}(\gamma(t))$ as $t\to\infty$. 
			Take $\epsilon\in (0,\beta-\alpha)$; using the notation above, we majorise \eqref{eq.xrepymultid} once again in a similar manner, this time arriving at   	
			\begin{equation*} 
				\frac{\|x(t)\|}{\gamma(t)} = \frac{\|y(t)\|}{\gamma(t)}+\frac{\|\Phi(t)\|}{\gamma(t)}\|\zeta\|+  D(\epsilon)\| I_d+A \|\frac{1}{\gamma_\epsilon(t)}\int_0^t e^{(\alpha+\epsilon)s}\|y(s)\|\,ds.
			\end{equation*} 
			Using L'H\^opital's rule once again, we see that the asymptotic relation $y(t)=\text{o}(\gamma(t))$ as $t\to\infty$ leads to the last term tending to zero as $t\to\infty$, so overall we have that $x(t)=o(\gamma(t))$ as $t\to\infty$. This contradicts our initial hypothesis, so the assumption that the second part of (A) is false must itself be false. Hence (A) is true, so (B) implies (A) as desired, and the equivalence is proven.  
		\end{proof}
		
		Rather than formulating and proving a cognate result in the cases where $\beta<\alpha$ and $\beta=\alpha$, which cover the situation where the perturbation is perturbation is subdominant, we will instead prove a results about Liapunov exponents, for which again we are able to characterise the asymptotic behaviour.
		or of a comparable order of magnitude (measured by the Liaupnov exponent).  
\subsection{Liapunov exponents} 
			We now seek a crude (but easily checkable) \textit{componentwise estimate} for perturbations which characterises the situation in which the upper Liapunov exponent of the unperturbed equation is not exceeded. To do this (a) we neither require refined conditions on the eigenstructure of $A$ nor a hypothesis on the normalising function $\gamma$ and (b) the conditions for the claimed growth bound are necessary and sufficient. We start by considering the unstable case, but later we will be able to remove entirely the assumption that the underlying unforced equation is unstable. 
				
				In order to get an equivalence involving the moving average of $F$, we need a preparatory scalar lemma.
				
				\begin{lemma} \label{lemma.liapunovexponent}
					Let $\alpha\geq 0$, $g\in C(\mathbb{R}^+;\mathbb{R})$. Let $u$ be the unique solution of the initial value problem  $u(0)=0$ $u'(t)=-u(t)+g(t)$ for $t\geq 0$. The following are equivalent:
					\begin{enumerate}
						\item[(A)] $\limsup_{t\to\infty} \log|u(t)|/t\leq \alpha$;
						\item[(B)] There exists $\Delta>0$, such that, for every $\epsilon>0$ there is $C=C(\epsilon,\Delta)>0$ such that 
						\[
						\left|\int_{(t-\theta)^+}^t g(s)\,ds\right|\leq C(\epsilon,\Delta) e^{(\alpha+\epsilon)t}, \quad \theta\in [0,\Delta], \quad t\geq 0.
						\]
					\end{enumerate}
				\end{lemma}
				\begin{proof}
					Assume (A). Then the continuity of $u$ implies that for every $\epsilon>0$ there is a $\bar{u}=\bar{u}(\epsilon)>0$ such that  $|u(t)|\leq \bar{u}(\epsilon)e^{(\alpha+\epsilon)t}$ for all $t\geq 0$. For $t\geq \theta$, we have the representation  
					\[
					\int_{(t-\theta)^+}^t g(s)\,ds = u(t)-u(t-\theta)+\int_{t-\theta}^t u(s)\,ds,
					\] 
					and for $t\in [0,\theta]$ we have 
					\[
					\int_{(t-\theta)^+}^t g(s)\,ds = u(t)+\int_{0}^t u(s)\,ds. 
					\]
					Then, for any $\theta\in [0,\Delta]$, it can be readily shown that an estimate of the form in (B) holds in both cases when $0\leq t\leq \theta$ and $t>\theta$. Thus (A) implies (B).
					
					Assuming (B), we apply Lemma~\ref{lemma.yintermsofftheta} as follows.  Let $\Delta>0$. For $\theta\in [0,\Delta]$, define  
					$g_\theta(t)=\int_{(t-\theta)^+}^t g(s)\,ds$ for $t\geq 0$. Define also  $\delta_g(t)=g(t)-g_\Delta(t)/\Delta$ for $t\geq 0$ and define  
					\begin{equation} \label{eq.G}
						G(t)=\int_0^t \delta_g(s)\,ds, \quad t\geq 0.
					\end{equation}
					Then by Lemma~\ref{lemma.yintermsofftheta},
					\begin{eqnarray} \label{eq.intdeltag1}
						G(t)&=& \frac{1}{\Delta}\int_0^\Delta g_\theta(t) \,d\theta, \quad t\geq \Delta,\\
						\label{eq.intdeltag2}
						G(t)&=& 
						=g_\Delta(t)-\frac{1}{\Delta}\int_0^t g_\Delta(s)\,ds, \quad t\in [0,\Delta],
					\end{eqnarray}
					and also by Lemma~\ref{lemma.yintermsofftheta}, we have that $u$ obeys 
					\begin{equation} \label{eq.urepave}
						u(t)=\frac{1}{\Delta}\int_0^t e^{-(t-s)}g_\Delta(s)\,ds + G(t) - \int_0^t e^{-(t-s)} G(s)\,ds, \quad t\geq 0.
					\end{equation}
					On the other hand, by (B) we have 
					\[
					|g_\theta(t)|\leq C(\epsilon,\Delta)e^{(\alpha+\epsilon)t}, \quad t\geq 0,
					\]
					and by \eqref{eq.intdeltag1} and \eqref{eq.intdeltag2} means that there exists a $C_2=C_2(\epsilon,\Delta)>0$ such that  
					\[
					|G(t)|\leq C_1(\epsilon,\Delta) e^{(\alpha+\epsilon)t}, \quad t\geq 0.
					\]
					Inserting these estimates into \eqref{eq.urepave} means that there exists a $C_3=C_3(\epsilon,\Delta)$ such that 
					\[
					|u(t)|\leq  C_3(\epsilon,\Delta) e^{(\alpha+\epsilon)t}, \quad t\geq 0.
					\]
					Finally, taking logarithms on both sides, dividing by $t$, and letting $t\to\infty$, we get 
					\[
					\limsup_{t\to\infty} \frac{1}{t}\log|u(t)|\leq \alpha+\epsilon.
					\]
					Since the left--hand side is independent of $\epsilon>0$, we may let $\epsilon\downarrow 0$ to get (A), as required. 
				\end{proof}
				
				With this lemma in hand, we have the following result, which connects the asymptotic behaviour of the time averages with the Liapunov exponent. 
				
				\begin{theorem} \label{thm.lepreservation}
					Suppose $F$ is continuous. Let $\lambda(A)=\alpha\geq 0$, and suppose $x$ is the unique continuous solution to \eqref{eq.xmult}. Then the following are equivalent
					\begin{enumerate}
						\item[(A)] For each $i\in \{1,\ldots,d\}$ 
						\[
						\limsup_{t\to\infty} \frac{1}{t}\log \left |e^{-t}\int_0^t e^sF_i(s)\,ds\right|\leq \alpha;
						\]
						\item[(B)] There exists $\Delta>0$ such that, for every $\epsilon>0$ there is a $C=C(\Delta,\epsilon)>0$ such that 
						\[
						\left\|
						\int_{(t-\theta)^+}^t F(s)\,ds\right\|\leq C(\epsilon,\Delta)e^{(\alpha+\epsilon)t}, \quad \theta\in [0,\Delta], \quad t\geq 0.
						\]
						\item[(C)] 
						\[
						\limsup_{t\to\infty} \frac{1}{t}\log \|x(t)\|\leq \alpha.
						\]
					\end{enumerate}
				\end{theorem} 
				We notice, by norm equivalence, and the fact that the choice of $\Delta$ is unimportant in our proofs (so we are free to choose $\Delta=1$) that the condition in (B) is equivalent to the componentwise condition
				\begin{itemize}
					\item[($\text{B}'$)] For each $i\in \{1,\ldots,d\}$ and for every $\epsilon>0$ there is a $C_i=C_i(\epsilon)>0$ such that 
					\[
					\left|
					\int_{(t-\theta)^+}^t F_i(s)\,ds\right|\leq C_i(\epsilon)e^{(\alpha+\epsilon)t}, \quad \theta\in [0,1], \quad t\geq 0.
					\]
				\end{itemize}
				Thus, in order to understand the circumstances under which the perturbed differential equation has the same upper Liapunov exponent as the preserved system, it is enough to perform an easier component--by--component check on the perturbation. This should give the interested reader inspiration to take on the programme suggested after Theorem~\ref{lemma.fthetaxsubexponentialexact}.
				\begin{proof}[Proof of Theorem~\ref{thm.lepreservation}]
					Suppose (C) holds. Then for every $\epsilon>0$, there exists a $C(\epsilon)>0$ such that $\|x(t)\|_1\leq C(\epsilon)e^{(\alpha + \epsilon)}$ for all $t\geq 0$.
					Now, take the triangle inequality across \eqref{eq.yrepxmultid}, and use the estimate on $x$ to obtain 
					\[
					\|y(t)\|_1\leq e^{-t}\|\zeta\|+C(\epsilon)e^{(\alpha + \epsilon)t} + C(\epsilon)\|I_d+A\|_1\int_0^t e^{-(t-s)}e^{(\alpha+\epsilon)s}\,ds.
					\]
					This immediately implies that there is a $K(\epsilon)>0$ such that $\|y(t)\|_1\leq K(\epsilon)e^{(\alpha +\epsilon)t}$ for all $t\geq 0$. Thus for each $i\in \{1,\ldots,d\}$, we have 
					\[
					\left| e^{-t}\int_0^t e^s F_i(s)\,ds\right|\leq \|y(t)\|_1
					\leq K(\epsilon)e^{(\alpha +\epsilon)t},
					\]
					which, upon taking logarithms, dividing by $t$ and letting $t\to\infty$, and then letting $\epsilon\downarrow 0$, yields (A). 
					
					On the other hand, if (A) holds, then for each $i\in \{1,\ldots,d\}$ and each $\epsilon>0$ there is a $T_i(\epsilon)>0$ such that 
					\[
					\left|e^{-t}\int_0^t e^s F_i(s)\,ds\right|\leq e^{(\alpha+\epsilon) t}, \quad t\geq T_i(\epsilon).
					\]
					Take $T=\max_{i=1,\ldots,d} T_i(\epsilon)$. Then 
					$\|y(t)\|_1\leq de^{(\alpha+\epsilon)t}$ for all $t\geq T(\epsilon)$. Since $y$ is continuous, there must exists a $K_1(\epsilon)>0$ such that $\|y(t)\|_1\leq K_1(\epsilon) e^{(\alpha+\epsilon)t}$ for all $t\geq 0$. On the other hand, since $\lambda(A)=\alpha$, for every $\epsilon>0$ there exists $K_2(\epsilon)>0$ such that 
					\[
					\|\Phi(t)\|_1 \leq K_2(\epsilon) e^{(\alpha+\epsilon/2)t}, \quad t\geq 0.
					\]
					Now majorising across \eqref{eq.xrepymultid} in the obvious way we get 
					\begin{multline*}
						\|x(t)\|_1\leq K_2(\epsilon)e^{(\alpha+\epsilon/2)t}\|\zeta\|+K_1(\epsilon)e^{(\alpha+\epsilon)t}\\
						+\|I_d+A\|_1 K_1(\epsilon)K_2(\epsilon)\int_0^t e^{(\alpha+\epsilon/2)(t-s)}  e^{(\alpha+\epsilon)s}\,ds.
					\end{multline*}
					Taking the larger exponent among the first two terms, and observing that $\int_0^t e^{\epsilon s}\,ds \leq e^{\epsilon t/2}/(\epsilon/2)$ for $t\geq 0$, we get 
					\[
					\|x(t)\|_1\leq (K_1(\epsilon)+K_2(\epsilon)\|\zeta\|_1)
					e^{(\alpha+\epsilon)t}+\|I_d+A\|_1 K_1(\epsilon)K_2(\epsilon) e^{(\alpha+\epsilon/2)t} \cdot \frac{e^{\epsilon t/2}}{\epsilon/2}.
					\]
					Hence $\|x(t)\|_1\leq K(\epsilon)e^{(\alpha+\epsilon)t}$ for $t\geq 0$, from which (C) readily follows. Thus (A) and (C) are equivalent. 
					
					We now show (A) and (B) are equivalent. If (A) holds, we notice that (A) is exactly the statement that the solution of the differential equation $y_i'(t)=-y_i(t)+F_i(t)$ for $t\geq 0$ with $y_i(0)=0$ obeys  $\limsup_{t\to\infty} \log |y_i(t)|/t\leq \alpha$. Applying Lemma \ref{lemma.liapunovexponent} to each component $y_i$ for $i\in \{1,\ldots,d\}$ this implies condition ($\text{B}'$) above. But using the $1$--norm, this implies (B). Conversely, if (B) holds, then ($\text{B}'$) holds, using the fact that $|v_i|\leq \|v\|_1$ for any vector $v=(v_1,\ldots,v_d)^T$. Now by Lemma~\ref{lemma.liapunovexponent} this gives 
					$\limsup_{t\to\infty} \log |y_i(t)|/t\leq \alpha$ for each $i\in\{1,\ldots,d\}$, which is precisely (A). 
				\end{proof}
				
				The reader may ask whether such a result, which characterises the preservation of the top Liapunov exponent of the unperturbed equation, is confined to the case when the Liapunov exponent $\alpha\geq 0$. The answer is that the result extends readily to deal with the case when $\alpha<0$ also, but some modifications to the proof are needed. We sketch the details and state precisely the theorem below. The reader will notice that the statement of the result is the same as the equivalence between conditions (B) and (C) in 
				Theorem~\ref{thm.lepreservation}, with the restriction on the sign of the Liapunov exponent removed, so the result in the general case is no weaker than in the case $\alpha\geq 0$. 
				\begin{theorem} \label{thm.lepreservationall}
					Suppose $F$ is continuous. Let $\lambda(A)=\alpha$, and suppose $x$ is the unique continuous solution to \eqref{eq.xmult}. Then the following are equivalent
					\begin{enumerate}
						\item[(A)] There exists $\Delta>0$ such that, for every $\epsilon>0$ there is a $C=C(\Delta,\epsilon)>0$ such that 
						\[
						\left\|
						\int_{(t-\theta)^+}^t F(s)\,ds\right\|\leq C(\epsilon,\Delta)e^{(\alpha+\epsilon)t}, \quad \theta\in [0,\Delta], \quad t\geq 0.
						\]
						\item[(B)] 
						\[
						\limsup_{t\to\infty} \frac{1}{t}\log \|x(t)\|\leq \alpha.
						\]
					\end{enumerate}
				\end{theorem} 
				Once again, we can replace $\Delta$ by unity, and use norm equivalence to simplify the checking of condition (A) to the equivalent condition below, where we need only check component--by--component.
				\begin{itemize}
					\item[($\text{A}'$)] For each $i\in \{1,\ldots,d\}$ and for every $\epsilon>0$ there is a $C_i=C_i(\epsilon)>0$ such that 
					\[
					\left|
					\int_{(t-\theta)^+}^t F_i(s)\,ds\right|\leq C_i(\epsilon)e^{(\alpha+\epsilon)t}, \quad \theta\in [0,1], \quad t\geq 0.
					\]
				\end{itemize} 
				\begin{proof}
					To show (B) implies (A), as on other occasions, we use the representation 
					\[
					\int_{(t-\theta)^+}^t F(s)\,ds = x(t)-x(\max(t-\theta,0))-A\int_{(t-\theta)^+}^t x(s)\,ds, \quad t\geq 0
					\]
					together with the estimate $\|x(t)\|\leq \bar{x}(\epsilon)e^{(\alpha+\epsilon)t}$ for $t\geq 0$, which holds for every $\epsilon>0$ and some $\bar{x}=\bar{x}(\epsilon)>0$. 
					
					To prove (A) implies (B), we need a new set of scalar auxiliary ODEs. For each $i\in \{1,\ldots,d\}$ let $y_i$ be the unique continuous solution of the ordinary differential equation $y_i'(t)=\alpha y_i(t)+F_i(t)$ for $t\geq 0$ with initial condition $y_i(0)=0$. Form the vector $y(t)=(y_1(t),\ldots,y_d(t))^T$. Then 
					$y'(t)=\alpha I_d y(t)+F(t)$ for $t\geq 0$ with $y(0)=0_d$. Assume temporarily that  
					\begin{equation} \label{eq.yilealpha}
						\limsup_{t\to\infty} \frac{1}{t}\log\|y_i(t)\|\leq \alpha, \quad i\in \{1,\ldots,d\},
					\end{equation}
					and notice it implies 
					\[
					\limsup_{t\to\infty} \frac{1}{t}\log\|y(t)\|\leq \alpha.
					\]
					Now, define $z(t)=x(t)-y(t)$ for $t\geq 0$. Then  $z'(t)=Az(t)+(A-\alpha I_d) y(t)$ for $t\geq 0$. Hence 
					\[
					x(t)=y(t)+z(t)=y(t)+\Phi(t)\xi+\int_{0}^t \Phi(t-s) (A-\alpha I_d) y(s)\,ds, \quad t\geq 0.
					\] 
					Reusing existing arguments, we see that this gives (B), so (A) implies (B), and the equivalence is proven, once we have established \eqref{eq.yilealpha}. This is equivalent to proving the following: suppose that $u$ is the unique continuous solution of $u'(t)=\alpha u(t)+g(t)$ for $t\geq 0$ with $u(0)=0$, and that there exists $\Delta>0$ such that for all $\epsilon>0$ there is a $C=C(\epsilon,\Delta)>0$ such that 
					\begin{equation} \label{eq.gestalpha}
						\left|\int_{(t-\theta)^+}^t g(s)\,ds\right|\leq C(\epsilon,\Delta)e^{(\alpha+\epsilon)t},\quad \theta\in [0,\Delta], \quad t\geq 0.
					\end{equation}
					Then $\limsup_{t\to\infty} \log|u(t)|/t\leq\alpha$. Using the notation from Lemma~\ref{lemma.liapunovexponent} for $g_\theta$, $G$ etc, we have that \eqref{eq.gestalpha} gives directly
					\[
					|g_\theta(t)|\leq C(\epsilon,\Delta)e^{(\alpha+\epsilon)t}, \quad t\geq 0.
					\]
					We can still apply \eqref{eq.intdeltag1} and \eqref{eq.intdeltag2} and this estimate to show that there is a  $C_1=C_1(\epsilon,\Delta)>0$ such that  
					\[
					|G(t)|\leq C_1(\epsilon,\Delta) e^{(\alpha+\epsilon)t}, \quad t\geq 0.
					\]
					Next using variation of constants, and the definition of $\delta_g$, we get  
					\[
					u(t)
					=e^{\alpha t}\int_0^t e^{-\alpha s}\delta_g(s)\,ds +\frac{1}{\Delta}\int_0^t e^{\alpha(t-s)}g_\Delta(s)\,ds.
					\]
					The exponential bound on $g_\Delta$ ensures the second term on the righthand side is bounded by a term of the form  $C_2(\epsilon,\Delta)e^{(\alpha+\epsilon)t}$. As for the first term, integrate by parts and recall that $G(t)=\int_0^t \delta_g(s)\,ds$ to get
					\[
					e^{\alpha t}\int_0^t e^{-\alpha s}\delta_g(s)\,ds
					= 	G(t) + \alpha 	e^{\alpha t}\int_0^t e^{-\alpha s} G(s)\,ds. 
					\] 
					By virtue of the bound on $G$, this term also enjoys a bound of the form $C_3(\epsilon,\Delta)e^{(\alpha+\epsilon)t}$. Therefore, overall, we have that for every $\epsilon>0$ there is $C'=C'(\Delta,\epsilon)>0$ such that $|u(t)|\leq C'(\Delta,\epsilon)e^{(\alpha+\epsilon)t}$ for all $t\geq 0$. From this $\limsup_{t\to\infty} \log|u(t)|/t\leq \alpha$ follows, as needed. 
				\end{proof}

				\bibliographystyle{unsrt}

			\end{document}